\documentclass[reqno]{amsart}
\usepackage{amsmath,amssymb,amscd,mathrsfs,enumerate,epsfig,curves,tikz-cd,color}

\numberwithin{equation}{section}

\newtheorem{theorem}[equation]{Theorem}
\newtheorem{proposition}[equation]{Proposition}
\newtheorem{lemma}[equation]{Lemma}
\newtheorem{corollary}[equation]{Corollary}

\theoremstyle{definition}

\newtheorem{example}[equation]{Example}

\theoremstyle{remark}

\DeclareMathOperator{\Av}{A\!v}
\DeclareMathOperator{\av}{av}
\DeclareMathOperator{\dist}{dist}
\DeclareMathOperator{\Dom}{\mathcal D}
\DeclareMathOperator{\Exp}{Exp}
\DeclareMathOperator{\Fix}{Fix}
\DeclareMathOperator{\frakD}{\mathfrak D}

\DeclareMathOperator{\op}{op}
\DeclareMathOperator{\spec}{spec}
\DeclareMathOperator{\supp}{supp}
\DeclareMathOperator{\rg}{rg}

\DeclareMathOperator{\sym}{\pmb \sigma}
\DeclareMathOperator{\tr}{tr}
\DeclareMathOperator{\Tr}{Tr}
\DeclareMathOperator{\WF}{WF}

\def\calA{\mathcal A}
\def\B{\mathcal B}
\def\E{\mathcal E}
\def\F{\mathcal F}
\def\G{\mathcal G}
\def\Hor{\mathcal H}
\def\calI{\mathcal I}

\def\calK{\mathcal K}
\def\Lie{\mathcal {L}}

\def\N{\mathcal N}

\def\T{\mathcal T}
\def\calU{\mathcal U}
\def\calV{\mathcal V}
\def\calW{\mathcal W}
\def\X{\mathcal X}
\def\Y{\mathcal Y}
\def\Z{\mathcal Z}

\def\Eigen{\mathscr E}

\def\Ha{\mathscr H}

\def\isotropy{\mathscr I}
\def\preisotropy{{\!\!\mathscr J\!}}
\def\Oh{O}

\def\C{\mathbb C}
\def\D{\mathbb D}
\def\Q{\mathbb Q}
\def\R{\mathbb R}

\def\a{\mathfrak a}
\def\A{\mathfrak A}

\def\h{\mathfrak h}
\def\m{\mathfrak m}
\def\proj{\mathfrak p}

\def\Dee{\mathrm {D}}

\def\gg{\mathrm g}
\def\Eta{\mathrm H}
\def\hor{\mathrm{Hor}}
\def\rmL{\mathrm L}
\def\rmR{\mathrm R}
\def\vert{\mathrm{Vert}}

\def\Laplacian{\square}

\def\im{i}
\def\minus{\backslash}
\def\Wedge{\raise2ex\hbox{$\mathchar"0356$}}

\def\inner{\mathbf i}
\def\Id{I}
\def\id{\mathrm{id}}

\def\rpar{)}
\def\lbra{[}

\def\st{\text{s.t.}}

\def\set#1{\{#1\}}
\def\contr{\rfloor}

\newcommand{\tensor}{\text{\raisebox{0.0ex}{\scalebox{0.8}{$\otimes$}}}}
\newcommand{\boxtensor}{\mathbin{\text{\raisebox{0.0ex}{\scalebox{0.9}{$\boxtimes$}}}}}
\def\trans{\pitchfork}

\def\one{{1\hspace{-3pt}1}}
\newcommand{\x}{\mathbin {\text{\raisebox{0.25ex}{\scalebox{0.6}{$\times$}}}} }

\def\display#1#2{\mbox{\parbox{#1} {#2}}}

\begin{document}

\title[Atiyah-Bott-Lefschetz formula]{An Atiyah-Bott formula for the Lefschetz number of a singular foliation
}
\author{Luiz Hartmann}
\email{hartmann@dm.ufscar.br}
\address{Department of Mathematics\\
Universidade Federal de São Carlos (UFSCar)\\
São Carlos, Brazil
}
\author{Gerardo A. Mendoza}
\email{gmendoza@temple.edu}
\address{Department of Mathematics\\
Temple University\\
Philadelphia, PA 19122}

\begin{abstract}
This paper presents a formula for the Lefschetz number of a geometric endomorphism in the style of the Atiyah-Bott theorem. The underlying data consist, first, of a compact manifold and a nowhere vanishing smooth real vector field $\T$ that preserves some Riemannian metric, and second, a sequence of first order operators on sections of Hermitian vector bundles with connection whose curvature is annihilated by $\T$ and for which parallel transport along integral curves of $\T$ is unitary. Assuming that the operators of the sequence commute with the various covariant derivatives $\Lie_\T=\nabla_\T$ and that their restriction to the spaces of sections annihilated by $\Lie_\T$ form a complex, an ellipticity condition gives finite-dimensionality of the resulting equivariant cohomology spaces. The Atiyah-Bott framework, adapted to give a geometric endomorphism only for the complex of $\Lie_\T$-parallel sections, together with the finiteness of cohomology allows for the definition of a Lefschetz number. Replacing the condition that the fixed points of the equivariant map $f$ associated with the endomorphism be simple by a condition on wave front sets, which is the underlying condition of Atiyah and Bott, yields that the set of closures of orbits by $\T$ left invariant by  $f$ is finite, and then a formula similar to theirs, now relating the Lefschetz number with traces along these orbits.
\end{abstract}

\thanks{Research partially supported by FAPESP 2019/08834-4.}
\maketitle

\section{Introduction}

The formula of Atiyah and Bott \cite{AtBo67,AtBo68} for the Lefschetz number has been the source of inspiration of many interesting extensions to new settings, applications, and intriguing connections to other fields. A sampling of variants include theorems in   
the context of quasicomplexes by Tarkhanov and Wallenta \cite{TaWa12} and Wallenta \cite{Wal14}, complex analysis/geometry as in the already cited work of Atiyah and Bott,  Donnelly and Fefferman \cite{DonFef86}, as well as Kytmanov et al. \cite{KyMyTa04}, and Tarkhanov \cite{Tar04}. Other extensions include work by \'Alvarez and Kordyukov \cite{AlKor08} and Heitsch and Lazarov \cite{HeLa}. Extensions to singular settings include manifolds with conical singularities studied by Bei \cite{Bei} and Nazaikinskii et al. \cite{NSSV}, and orbifolds, for instance Sardo Infirri \cite{SI92} in the complex case. In another direction, there are Atiyah-Bott type formulas in the case of flows \cite{Dei97,OBr77}. For connections to other fields see Leichtnam \cite{Leichtnam}.

Closer to the present work, Zeggar \cite{Zeg92} studied the Atiyah-Bott formula for the Lefschetz number within the framework of basic cohomology for a regularly oriented Riemannian foliation adapting earlier work of Toledo \cite{Tol73}. Such foliations include orbifolds as base spaces (as the space of leaves). Still in the case of regular foliations but in a somewhat different direction, \cite{AlKor08} and \cite{HeLa} deal with a kind of families Atiyah-Bott-Lefschetz formula in that they consider fiberwise complexes for which they derive, under certain hypotheses, a global integral formula.

\medskip
Our starting point is a closed $n$-manifold $\N$ together with a nowhere vanishing real vector field $\T$ that admits an invariant Riemannian metric. These (man\-i\-fold-and-vector field) arise naturally as boundaries of complex $b$-manifolds with compact boundary, see \cite[\S 4]{Me2014}. With additional structure on the boundary inherited from the interior complex $b$-structure, they can be viewed as generalizations of holomorphic line bundles. Ultimately it is these boundaries of complex $b$-manifolds, which admit families of CR structures, that are our interest. Complex $b$-manifolds are of interest in that they generalize real blowups of complex analytic manifolds with point singularities. Further work on such manifolds will be taken up elsewhere.

Given $\N$ and $\T$ (and some invariant Riemannian metric $\gg$), all fixed henceforth, denote by $\a_t$ the one-parameter group of diffeomorphisms generated by $\T$. Since these are isometries, the closure of $\mathcal G=\set{\a_t:t\in \R}$, denoted by $G$, is a Lie subgroup of the group of all isometries of $\N$ for that metric. The group $G$ is compact and abelian because $\mathcal G$ is abelian and the group of isometries of $\gg$ is compact. It is connected because it is the closure of a connected set. We will write $\a_g$ for the isometry defined by $g\in G$ while still keeping the notation $\a_t$ for the diffeomorphisms generated by $\T$, and use multiplicative notation for the group operation, so $\a_{g_1}\circ\a_{g_2}=\a_{g_1g_2}$. 

The group $G$ is diffeomorphic to a torus, hence its orbits are diffeomorphic to tori of possibly varying dimensions. These orbits define the foliation mentioned above. That this foliation is Riemannian is a consequence of a theorem of Molino  \cite[Prop. 6.2, p. 191]{Molino88} stating that the singular foliation obtained by closing the leaves of a Riemannian foliation is again Riemannian. In our case, the initial foliation is by the orbits of $\T$. More details of this initial description will be given later. We point out that the specific invariant Riemannian metric on $\N$ is not important, only the existence of one such metric matters; $G$ itself is independent of the metric.

We will write $\B$ for the space of orbits of $G$ with the quotient topology, use $\Y$ to denote an arbitrary element of $\B$, and $\Y_p$ for the orbit of $G$ through $p$. Note that $\Y_p$ is the closure of the orbit of $p$ by $\mathcal G$. The space $\B$ is Hausdorff, not necessarily a manifold, but has an open dense subset that is a manifold. Generally we will not emphasize the fact that the results may have an interpretation as theorems on a manifold with singularities. However, we will use notation such as $C^\infty(\B;\E)$ to mean certain spaces of smooth sections on $\N$. For instance, $C^\infty(\B)$ is the space of smooth function on $\N$ which are constant on orbits of $G$; this point of view, akin to an analytic desingularization of $\B$, is not uncommon in the present context in the literature.

\medskip

Let 
\begin{equation}\label{Inner}
\inner_\T:\Wedge^q\N\to\Wedge^{q-1}\N
\end{equation}
be interior multiplication by $\T$. The kernel of $\inner_\T$ in $\C T^*\N$ is a subbundle $\Hor^*$ of rank $n-1$, and the kernel of \eqref{Inner} is $\Wedge^q \Hor^*$ (we slightly simplify the notation by regarding $\Wedge^q\N$ already as the exterior powers of the complexification of $T^*\N$). The central element of our theorem is a sequence of first order differential operators
\begin{equation*}
C^\infty(\N;E^0)\xrightarrow{\Dee_0}C^\infty(\N;E^1)\xrightarrow{\Dee_1}\cdots \xrightarrow{\Dee_{m-1}}C^\infty(\N;E^m)
\end{equation*}
on sections of Hermitian vector bundles over $\N$, each equipped with an Hermitian connection $\nabla^q$ whose curvature is $E^q\otimes \Wedge^2\Hor^*$-valued. The sequence is not required to be a cochain complex: $\Dee_{q+1}\circ \Dee_q$ may not vanish. We fix once and for all a smooth positive density $\m$ invariant under $\a_t$ and use it and the Hermitian inner product to give each of the spaces of the complex a pre-Hilbert space structure.

We write $\Lie_\T$ for $\nabla^q_\T$ leaving implicit the reference to $q$. Since the connection is Hermitian,
\begin{equation*}
\T h_q(u,v)=h_q(\Lie_\T u,v)+h_q(u,\Lie_\T v),\quad u,v\in C^\infty(\N;E^q),
\end{equation*}
where $h_q$ is the Hermitian metric of $E^q$. This and the $\T$-invariance of $\m$ give that
\begin{equation}\label{iTisSymmetric}
-\im\Lie_\T:C^\infty(\N;E^q)\to C^\infty(\N;E^q) \text{ is formally selfadjoint}.
\end{equation}

We further require that
\begin{equation}\label{Commutators}
\Lie_\T \circ \Dee_q=\Dee_q\circ \Lie_\T \text{ and }\Dee_{q+1}\circ \Dee_q u=0 \text{ if }\Lie_\T u=0.
\end{equation}
 for each $q$. Define
\begin{equation}\label{PseudoSections} 
C^\infty(\B;\E^q)=\set{u\in C^\infty(\N;E^q):\Lie_\T u=0}.
\end{equation}
Letting now $\D^q$ denote the restriction of $\Dee_q$ to $\E^q$ we get a cochain complex
\begin{equation}\label{TheGeneralComplex}
C^\infty(\B; \E^0)\xrightarrow{\D_0}C^\infty(\B; \E^1)\xrightarrow{\D_1}\cdots \xrightarrow{\D_{m-1}}C^\infty(\B; \E^m)
\end{equation}
whose cohomology spaces will be denoted $H^q(\B,\E)$. These may not be finite-{}di\-men\-sional. However, we have:

\begin{proposition}\label{Ellipticity}
If
\begin{equation}\label{FiniteCohomology}
P_q = \Dee_{q}^\star \circ \Dee_{q}+\Dee_{q-1}\circ\Dee_{q-1}^\star-\Lie_\T^2
\end{equation}
is elliptic, then the cohomology spaces of the complex \eqref{TheGeneralComplex} are finite-di\-men\-sion\-al.
\end{proposition}

The $\star$ denotes the formal adjoint.  The proof consists of showing that the restriction of $P_q$ to $C^\infty(\B;\E^q)$ is the Laplacian in degree $q$ for the complex \eqref{TheGeneralComplex} and relating this to cohomology using Hodge theory. This does require some delicate analysis of the domains of the operators involved; the proof is given in Section \ref{s-FiniteDimCoho}. For now we give credence to this approach by noting  that \eqref{iTisSymmetric} and \eqref{Commutators} imply that $
\Dee_q$ and $\Lie_\T$ commute, so $\D_q^\star$ maps $C^\infty(\B;E^{q+1})$ to $C^\infty(\B;E^q)$.

\medskip
Next we adapt Atiyah-Bott's \cite{AtBo67} notion of geometric endomorphism to the present situation. Using parallel transport along the curves $t\mapsto \a_t(p)$ we get a one-parameter group of unitary morphisms $\A^q_t:E^q\to E^q$ covering $\a_t$. 

\begin{proposition}\label{OnConnections}
The closure of $\set{\A^q_t:t\in \R}$ in the compact-open topology of vector bundle morphisms $E^q\to E^q$ is again a compact connected abelian Lie group, $\hat G^q$, acting on $E^q$ by smooth unitary morphisms $\hat g\mapsto \A^q_{\hat g}$. For each $q$ there is a surjective homomorphism $\wp^q:\hat G^q\to G$ such that the homomorphism $\A^q_{\hat g}$ covers $\a_{\wp^q(\hat g)}$.
\end{proposition}

The validity of the proposition for the vector bundle $E^q$ is equivalent to the assumed existence of a Hermitian connection whose curvature is valued in $E^q\otimes \Wedge^2\Hor^*$. The proofs of this statement and the proposition are given in Section \ref{ConnectionsGroupsAverage}. Appendix~\ref{CurvatureAndGroups} gives a full discussion on the existence of such connections. The surjectivity of $\wp^q$ is just a consequence of the fact that the bundle morphisms $\A^q_t$ cover $\a_t$.

\medskip
Define $\pmb \A^q_t:C^\infty(\N;E^q)\to C^\infty(\N;E^q)$ by the formula
\begin{equation}\label{AonSections}
\big(\pmb \A^q_t(u)\big)(p)=\A^q_t\big(u(\a_{-t}(p))\big).
\end{equation}

\medskip
Assume given a smooth map $f:\N\to \N$ such that $f_*\T=\T$ and vector bundle morphisms
\begin{equation}\label{GeometricEndomorphisms}
\phi^q:f^*E^q\to E^q
\end{equation}
covering the identity such that, with $f^*\A^q_t:f^*E^q\to f^*E^q$ defined by $
(f^* \A^q_t)(p,\eta)=\big(\a_t(p), \A^q_t(\eta)\big)$ for $\eta\in E^q_{f(p)}$, the  diagram
\begin{equation}\label{CDforPhi}
\begin{tikzcd}
f^*E^q \arrow[d,"f^*\A^q_t"'] \arrow[r,"\phi^q"]& E^q \arrow[d, "\A^q_t"]\\
f^*E^q\arrow[r,"\phi^q"]& E^q\\
\end{tikzcd}
\end{equation}
commutes. 
Then the maps 
\begin{equation}\label{DefOfPhiq}
\Phi^q:C^\infty(\N;E^q)\to C^\infty(\N;E^q),\quad \Phi^q(u)(p)=\phi^q(p,u(f(p)))
\end{equation}
commute with $\pmb \A_t^q$, hence they commute with $\Lie_\T$. Thus they map $C^\infty(\B;\E^q)$ to itself. Adapting from Atiyah-Bott \cite{AtBo67}, the setup is then completed with the requirement that when assembled in the diagram 
\begin{equation}\label{CD}
\begin{CD}
\cdots & \longrightarrow\ &C^\infty(\B; \E^{q-1})&@>{\D_{q-1}}>> &C^\infty(\B; \E^q)&@>{\D_{q}}>> &C^\infty(\B; \E^{q+1})&\ \longrightarrow &\ \cdots\\
& &@V\Phi^{q-1}VV & & @V\Phi^qVV & & @V\Phi^{q+1}VV \\
\cdots & \longrightarrow\ &C^\infty(\B; \E^{q-1})&@>{\D_{q-1}}>> &C^\infty(\B; \E^q)&@>{\D_{q}}>> &C^\infty(\B; \E^{q+1})&\ \longrightarrow &\ \cdots
\end{CD}
\end{equation}
they give a cochain map. 

Denote by $\pmb \Phi^q$ the maps induced in cohomology. Assuming the hypothesis of Proposition \ref{Ellipticity}, there is a well defined Lefschetz number,
\begin{equation}\label{LefschetzNumber}
L_{f,\phi}=\sum_{q=0}^m (-1)^q\tr \pmb \Phi^q.
\end{equation}

\medskip
The last elements needed in the statement of the main theorem are essentially analytic in nature. To bring them in we need to define averaging operators. In the diagram 
\begin{equation}\label{CDforAv}
\begin{tikzcd}
E^q \arrow[d]&\hat \pi^*E^q \arrow[l,"\hat{\pmb \rho}"']\arrow[d] \arrow[r,"\hat{\pmb \pi}"]& E^q \arrow[d]\\
\N & \N\times \hat G^q \arrow[l,"\hat\rho"']\arrow[r,"\hat\pi"]& \N\\
\end{tikzcd}
\end{equation}
the vertical arrows and the maps $\hat \pi$ and $\hat{\pmb \pi}$ are the canonical projections. The map $\hat{\pmb \rho}$ is the bundle morphism
\begin{equation*}
\hat{\pmb \rho}(p,\hat g;\eta)=\A^q_{\hat g}(\eta),\qquad (p,\hat g;\eta)\in \hat\pi^*E^q
\end{equation*}
covering
\begin{equation*}
\hat \rho(p,\hat g)=\a_{\wp^q(\hat g)}(p).
\end{equation*}
Observe that $\hat \rho:\N\times\hat G^q\to \N$ is surjective, in fact makes $\N\times\hat G^q$ into a (trivial) principal $\hat G^q$-bundle, $(p,\hat g)\cdot \hat h=(\a_{\wp^q(h^{-1})}p,\hat g h)$. The fiber over $p\in \N$ is
\begin{equation*}
\hat \rho^{-1}(p)=\set{(\a_{\wp^q(\hat g^{-1})}(p),\hat g):\hat g\in \hat G^q}.
\end{equation*}
Define 
\begin{equation}\label{Av}
\Av^q: C^\infty(\N;E^q)\to C^\infty(\N;E^q),\quad \Av^q(u)=\hat{\pmb \rho}_*((\hat{\pmb \pi}^* u)\, \hat \mu^q)
\end{equation}
where $\hat \mu^q$ the normalized Haar measure on $\hat G^q$. Explicitly, if $(\a_{\wp^q(\hat g^{-1})}(p),\hat g;\eta)\in \pi^* E^q$, thus $\eta\in E^q_{\a_{\wp^q(\hat g^{-1})}(p)}$, then $\A^q_{\hat g}(\eta)\in E^q_p$. Thus if $u$ is a continuous section of $E^q$ along $\Y_p$, then 
\begin{equation*}
\hat G^q\ni \hat g\mapsto \A^q_{\hat g}(u(\a_{\wp^q(\hat g^{-1})}p))\in E^q_p
\end{equation*}
is continuous and we can define 
\begin{equation}\label{defOfAverage}
(\Av^q u)(p)=\int_{\hat G^q} \A^q_{\smash[t]{\hat g}} \big(u(\a_{\wp^q(\hat g^{-1})}(p))\big)\,d\hat \mu^q(\hat g).
\end{equation}
Of course this averaging procedure can be applied to $u$ also if it is defined on all of $\N$. By properties of wavefront set under pull-back and push-forward, if $u$ is smooth, then so is $\Av^q(u)$. 

\begin{lemma}
The operator in \eqref{Av} satisfies
\begin{equation*}
\A^q_{\hat g}\big(\Av^q(u)(p)\big)=\Av^q(u)(\a_{\wp^q({\hat g})}(p))\text{ for all }\hat g\in \hat G^q.
\end{equation*}
In particular
\begin{equation*}
\A^q_{-t}(\Av^q(u)(\a_t(p)))=\Av^q(u)(p)
\end{equation*}
and so $\Lie_\T\Av^q u=0$. It follows that $\Av^q$ is a projection on $C^\infty(\B;\E^q)\subset C^\infty(\N;E^q)$.
 \end{lemma}

The proof is clear. In fact $\Av^q$ is an orthogonal projection, see Lemma \ref{AvIsProjector} (also a Radon transform, since \eqref{CDforAv} is a double fibration). The relevancy of the averaging operators lies in that they function as replacements of the diagonal embedding in the original version of the Atiyah-Bott formula. 

The operators $\Av^q$ and $\Phi^q\circ \Av^q$ have respective Schwartz kernels
\begin{equation*}
K_{\Av^q},\ K_{\Phi^q\circ \Av} \in C^{-\infty}(\N \times \N;E^q\boxtensor E^{q,*})
\end{equation*}
(we trivialize the density bundle using the density $\m$). We describe their wavefront sets with the aid of the mappings
\begin{equation}\label{iota_f}
\begin{gathered}
\hat \iota_f:\N\times \hat G^q\to \N \times \N,\quad \iota_\Delta:\N\to\N\times \N\\
\hat \iota_f(p,\hat g)=(\a_{\wp^q(\hat g)}(p),f(p)),\quad \iota_\Delta(p)=(p,p).
\end{gathered}
\end{equation}
For an arbitrary smooth map $\iota:\X\to\Y$ between some manifolds $\X$, $\Y$ define 
\begin{equation*}
N^*(\iota)=\set{\eta\in T^*\Y:\exists x\in \X\ \st\ \eta\in T^*_{\iota(x)}\Y,\text{ and }\iota^*\eta=0}.
\end{equation*}
Thus if $\iota$ is an embedding, then $N^*(\iota)$ is the conormal bundle of $\iota(\X)$. In the present situation the map $\hat \iota_f$ need not be an embedding.

\begin{lemma}
The wave front set of $K_{\Phi^q\circ \Av^q}$ consists of the nonzero elements of $N^*(\hat \iota_f)$. 
\end{lemma}

The proof is simplified by reduction to the scalar case. The bottom row of \eqref{CDforAv} factors through the analogous diagram in the scalar case (where $\hat G^q$ is replaced by $G$ and $E^q$ be the trivial line bundle),
\begin{equation}\label{ScalarAv}
\begin{tikzcd}
\N & \N\times \hat G^q \arrow[l,"\hat \rho"']\arrow[d,"\id\times \hat\wp"]\arrow[r,"\hat \pi"]&\N\\
& \arrow[lu,"\rho"]\N\times G \arrow[ru,"\pi"'] &
\end{tikzcd}
\end{equation}
with $\pi(p,g)=p$ and $\rho(p,g)=\a_g(p)$. Then the wave front set of the Schwartz kernel of 
\begin{equation*}
\av:C^\infty(\N)\to C^\infty(\N),\quad \av(u)=\rho_*(\pi^*u\,\mu)
\end{equation*}
can be found (easily but tediously) via standard calculations using pull-backs and push-forwards from H\"ormander \cite{Hor71}, likewise that of $f^*\circ \av$. (Here $\mu$ is the normalized Haar measure of $G$.) Focusing on the later, one sees that $\hat \iota_f$ factors as 
\begin{equation*}
\N\times \hat G^q\to \N\times G\to \N \times \N
\end{equation*}
where the first map is $(p,\hat g)=(p,\wp^q(\hat g))$, a submersion, and the second is
\begin{equation*}
\iota_f:\N\times G\to\N,\quad \iota_f(p,g)=(\a_g(p),f(p)),
\end{equation*}
the scalar version of $\hat \iota_f$. Consequently $N^*(\hat \iota_f)=N^*(\iota_f)$. The details are given in Section \ref{WaveFront}.

If $G$ is the trivial group, then $\iota_f$ is the embedding of $\N$ in $\N \times \N$ as the graph of $f$. In this case,
\begin{equation}\label{transversality}
N^*(\iota_f^*)\cap N^*(\iota_\Delta^*) \text{ contains no zero covectors}. 
\end{equation}
is the Atiyah-Bott condition of simplicity of the fixed points of $f$, equivalently, the transversality in $\N \times \N$ of the graph of $f$ and the diagonal. In the present case this is also the condition leading to a formula. Since  $\a_t\circ  f=f\circ \a_t$ for all $t$, $f$ maps orbits of $G$ to orbits of $G$. (Note that, in general, $f(\Y_p)=\Y_{f(p)}$ because the curve $t\mapsto f(\a_t(p))$ is dense in $\Y_{f(p)}$.) Let
\begin{equation*}
\pmb f:\B\to \B
\end{equation*}
denote the corresponding map, which is continuous. Let $\Fix(\pmb f)$ be the set of fixed points of $\pmb f$,
\begin{equation*}
\Fix(\pmb f)=\set{\Y\in \B: f\text{ maps }\Y \text{ to itself}}.
\end{equation*}
It is not required that $f$ fixes $\Y$ pointwise, only that $f(\Y)=\Y$. 

\begin{theorem}
If \eqref{transversality} holds, then $\Fix(\pmb f)$ is a discrete set.
\end{theorem}

The proof is given in Section \ref{discreteness}. 

\smallskip
Let $\Y\in \Fix(\pmb f)$, pick $p_0\in \Y$. Then $f(p_0)\in \Y$ implies $f(p_0)=\a_{g_0^{-1}}(p_0)$ for some $g_0\in G$ and it follows that $f(p)=\a_{g_0^{-1}}(p)$ for every $p\in \Y$, that is, $\a_{g_0}\circ f$ is the identity on $\Y$. If this is the case, then
\begin{equation*}
d(\a_{g_0}\circ f)\big|_{N^*_{p}\Y}:N^*_p\Y\to N^*_p\Y
\end{equation*}
for any $g\in G$ and it makes sense to ask whether 
\begin{equation}\label{SimpleFixedPoints}
d(\a_{g_0}\circ f)\big|_{N^*_p\Y}-\Id:N^*_p\Y\to N^*_p\Y\text{ has trivial kernel}.
\end{equation}

\begin{proposition}
Condition \eqref{transversality} holds if and only if for every $\Y\in \Fix(\pmb f)$, \eqref{SimpleFixedPoints} holds for every $g_0\in G$ such that $\a_{g_0}^{-1}\circ f$ is the identity on $\Y$.
\end{proposition}

\begin{theorem}\label{TheoremA}
Suppose \eqref{transversality} holds. Then $\Fix(\pmb f)$ is finite and, in each degree, the pull-back of the Schwartz kernel $K_{\Phi^q\circ \Av^q}$ of $\Phi^q\circ \Av^q$ by $\iota_\Delta$ is well defined. The resulting distribution
\begin{equation*}
\iota_\Delta^* (K_{\Phi^q\circ \Av^q})\in C^{-\infty}(\N; \iota_\Delta^* (E^q\boxtensor {E^q}^*))
\end{equation*}
is of order $0$ and  supported on 
\begin{equation*}
\bigcup_{\Y\in \Fix(\pmb f)} \Y.
\end{equation*}
Using the canonical section $\calI_\tensor^q\in  C^\infty(\N;{E^q}^*\otimes E^q)$ that corresponds to the identity map $\Id^q:{E^q}^*\to {E^q}^*$ we have
\begin{equation}\label{AbstractRHSFormula}
L_{f,\phi}=\sum_{q=0}^m (-1)^q \langle \iota_\Delta^*(K_{\Phi^q\circ \Av}^q), \calI_\tensor^q\rangle.
\end{equation}
\end{theorem}

For the proof see Section \ref{Limits}.

\medskip
The final step is the computation of the pairings in \eqref{AbstractRHSFormula}, completed in Section~\ref{sLimit}, see Theorem \ref{TheoremC}. 

Choose, for each $\Y\in \Fix(\pmb f)$ and $q = {0,1,\dots,m}$, an element $\hat g_\Y^q\in \hat G^q$ such that $\a_{\wp^q(\hat g_\Y^q)}\circ f$ fixes every point of $\Y$. This is possible precisely because $\Y\in \Fix(\pmb f)$. The map $\phi^q:f^*E^q\to E^q$ gives a map $\phi^q(p):E^q_{f(p)}\to E^q_p$, so $\phi^q(p) \circ \A_{(\hat g_\Y^q)^{-1}}^q$ maps $E^q_p$ to itself for any $p\in \Y$. Its trace is independent of $p\in \Y$ and the choice of $\hat g_\Y^q$ (subject to $\a_{\wp^q(\hat g_\Y^q)}\circ f$ being the identity on $\Y$). The isotropy group $\isotropy_\Y\subset G$ of the action of $G$ on $\N$ at a point of $\Y$ is independent of the point; let $\preisotropy_\Y^q\subset \hat G^q$ be the preimage of $\isotropy_\Y$ by $\wp^q$; this is a compact subgroup, not necessarily connected. Let $\hat G_\Y^q\subset \hat G^q$ be a compact connected subgroup of minimal dimension transverse to $\preisotropy_\Y^q$ at the identity element. The normalized Haar measures $\hat \mu_{\hat G^q}$ of $\hat G^q$ and $\hat \mu_{\preisotropy_\Y^q}$ of $\preisotropy_\Y^q$ determine a Haar measure $\hat \mu_{\hat G_\Y^q}$ on $\hat G_\Y^q$ via the requirement that $\hat \mu_{\hat G^q}= \hat \mu_{\preisotropy_\Y^q}\otimes \hat \mu_{\hat G_\Y^q}$ near the identity. For any fixed $p\in \Y$, 
\begin{equation*}
G_\Y^q\ni \hat g \mapsto \a_{\wp^q(\hat g)}(p)\in \Y
\end{equation*}
is a finitely sheeted covering. Let $\mathfrak k_\Y^q$ be the number of sheets. With this notation, Theorem~\ref{TheoremC} gives

\begin{theorem}\label{TheoremB} The Lefschetz number \eqref{LefschetzNumber} of the cochain map \eqref{CD} is
\begin{equation*}
L_{f,\phi} =\sum_{\Y\in \Fix(\pmb f)} \sum_{q=0}^m(-1)^q\frac{\hat \mu_{\hat G_{\Y}^q}(\hat G_{\Y}^q)}{\mathfrak k_\Y^q}\int_{\preisotropy_\Y^q}\frac{\tr\big(\phi^q(p_\Y)\circ \A_{\hat h(\hat g_\Y^q)^{-1}\big)}^q}{\big|\det\big[(\a_{\wp(\hat g_\Y^q\hat h^{-1})}\circ f)^*\big|_{N^*_{p_\Y}\Y}-\Id\big]\big|}\, d\hat \mu_{\preisotropy_\Y^q} (\hat h).
\end{equation*}
Here $p_\Y$ is some arbitrarily fixed point of $\Y$.
\end{theorem}

\section{A model in de Rham cohomology}

Recall that $\Wedge^q \Hor^*\subset \Wedge^q \N$ is the kernel of interior multiplication by $\T$,
\begin{equation*}
\inner_\T:\Wedge^q\N\to\Wedge^{q-1}\N
\end{equation*}
be interior multiplication by $\T$. Assuming without loss of generality that the invariant Riemannian metric $\gg$ satisfies $\gg(\T,\T)=1$, define the one-form $\theta$ by
\begin{equation}\label{PseudoConnection}
\langle \theta, v\rangle=\gg(\T(p),v), \quad v\in T_p\N.
\end{equation}
It satisfies $\inner_\T\theta=1$, $\Lie_\T\theta=0$ and $\inner_\T d\theta=0$. The map 
\begin{equation*}
\pi:\Wedge^q\N\to\Wedge^q\N, \quad \pi \eta= \eta-\theta\wedge \inner_\T \eta
\end{equation*}
is the projection on $\Wedge^q\Hor^*$ with kernel $\theta\wedge \Wedge^{q-1}\N$. With the maps
\begin{equation*}
\Dee_q:C^\infty(\N;\Wedge^q\Hor^*)\to C^\infty(\N;\Wedge^{q+1}\Hor^*),\quad \Dee_q\psi=\pi d\psi.
\end{equation*}
we get a sequence 
\begin{equation}\label{HorSequence}
C^\infty(\N)\xrightarrow{\Dee_0}C^\infty(\N;\Wedge^1\Hor^*)\xrightarrow{\Dee_1}\cdots \xrightarrow{\Dee_{n-2}}C^\infty(\N;\Wedge^{n-1}\Hor^*)
\end{equation}
which need not be a complex. The operator $D_0$ resembles a connection and the form $\theta$ is like the connection form on the bundle of unit frames (a circle bundle) of a Hermitian connection on a complex line bundle over $\B$. With some compatible almost CR  structure on $\N$ it is a pseudo-hermitian connection as introduced by Webster \cite{Webster1978}.

We use the metric $\gg$ to give $\Wedge^q\Hor^*$ a Hermitian metric. We will show that
\begin{equation*}
P_q=\Dee_{q}^\star \circ \Dee_{q}+\Dee_{q-1}\circ\Dee_{q-1}^\star-\Lie_\T^2
\end{equation*}
is elliptic, cf. Proposition \ref{Ellipticity}. Here $\Lie_\T$ is the actual Lie derivative with respect to $\T$ on differential forms on $\N$. To show that $P_q$ is elliptic we show that the symbol sequence of \eqref{HorSequence} is exact in the complement of the span of $\theta$. Namely, let $f:\N\to \R$ be some smooth function. Then, at any $p\in \N$ where $df(p)\ne 0$ we have
\begin{equation*}
\lim_{s\to\infty}\frac{1}{s}e^{-\im s f}\Dee_q(e^{\im s f}u)=\im (df-\T f\,\theta)\wedge u, \quad u\in C^\infty(\N,\Wedge^q\Hor^*),
\end{equation*}
so
\begin{equation*}
\sym(\Dee_q)(\pmb \xi)(u)=\im \big(\pmb \xi -\langle \pmb \xi,\T\rangle\theta\big)\wedge u.
\end{equation*}
From this formula we see that 
\begin{equation}
\cdots\to \Wedge^{q-1}\Hor_p^*\xrightarrow{\sym(\Dee_{q-1})(\pmb \xi)}\Wedge^q\Hor_p^*\xrightarrow{\sym(\Dee_{q+1})(\pmb\xi)}\Wedge^{q+1}\Hor_p^*\to \cdots
\end{equation}
is exact when $\pmb \xi -\langle \pmb \xi,\T\rangle\theta\ne 0$ with the same argument as in the case of the actual de Rham sequence. Thus the Laplacians of \eqref{HorSequence} fail to be elliptic on the span of $\theta$. Adding the term $-\Lie_\T^2$ gives ellipticity:
\begin{equation*}
\sym(P_q)(\pmb \xi)=\|\pmb \xi -\langle \pmb \xi,\T\rangle\,\theta\|^2\Id +\langle \pmb \xi,\T\rangle^2\Id.
\end{equation*}

Define
\begin{equation}\label{PseudoForms}
C^\infty(\B;\E^q)=\set{u\in C^\infty(\N;\Wedge^q\Hor^*):\Lie_\T u=0}.
\end{equation}
Since $d\circ\Lie_\T=\Lie_T \circ d$, $du\in C^\infty(\B;\E^{q+1})$ if $u\in C^\infty(\B;\E^q)$. This space consists of basic forms \cite[p. 38]{Molino88} with respect to the foliation determined by $\T$, but not necessarily of the singular foliation determined by the orbits of $G$. With the notation $\D_q u $ in place of $du$ we have a complex
\begin{equation}\label{TheHComplex}
C^\infty(\B; \E^0)\xrightarrow{\D_0}C^\infty(\B; \E^1)\xrightarrow{\D_1}\cdots \xrightarrow{\D_{n-2}}C^\infty(\B; \E^{n-1})
\end{equation}
whose cohomology groups are finite-dimensional because of Proposition \ref{Ellipticity}.

A smooth map $f:\N\to \N$ such that $f_*\T=\T$ induces by pull-back geometric endomorphisms with the right equivariant properties. The condition for the validity of the Atiyah-Bott formula for the Lefschetz number of $f$ is a condition on $f$, so Theorems \ref{TheoremA} and \ref{TheoremB} are valid.

\begin{example}\label{ActualLefschetz} Suppose that $\B$ is a smooth closed manifold and $\N=\B\times S^1$. The vector field $\T$ is the standard angular derivative and the metric on $\N$ is a product metric. The map $\N\to \B$ is the canonical projection, the spaces $C^\infty(\B; \E^q)$ are canonically $C^\infty(\B;\Wedge^q\B)$, and $\D$ is just the de Rham differential on $\B$. In the Atiyah-Bott formula for the Lefschetz number the condition on $\pmb f$ is that the fixed points are simple. This translates into the condition that the graph $\pmb \Gamma$ of $\pmb f$ in $\B\times \B$ is transverse to the diagonal $\pmb\Delta\subset \B\times \B$,  equivalently, that the conormal bundles of these two submanifolds intersect only at zero. From the operator point of view, the wave front sets of  the Schwartz kernels of the operators $\pmb f^*$ and $\pmb \iota^*$ (the embedding of the diagonal) do not intersect. This is the interpretation of the transversality condition that, with suitable modifications, will work in our case.
\end{example}


\section{Finite dimensionality of cohomology}\label{s-FiniteDimCoho}

In the rest of the paper we work with the general setup of the introduction. We now assume that 
\begin{equation}\label{DeltaPlusL2}
P_q=\Dee_{q}^\star \circ \Dee_{q}+\Dee_{q-1}\circ\Dee_{q-1}^\star-\Lie_\T^2\text{ on }C^\infty(\N;E^q)\text{ is elliptic}
\end{equation}
and prove that the cohomology groups of the complex \eqref{TheGeneralComplex} are finite-dimensional. As already mentioned this will be a consequence of the Hodge theory for the complex \eqref{TheGeneralComplex} and will follow from a careful analysis of the domains of the operators involved. 

Recalling that $C^\infty(\N;E^q)$ is a pre-Hilbert space, let $L^2(\N;E^q)$ denote its completion and let $L^2(\B;\E^q)\subset L^2(\N;E^q)$ denote the kernel of $\Lie_\T$. 

We give $\Dee$ and $\Lie_\T$ their maximal domains:
\begin{gather*}
\Dom^q(\Dee)=\set{u\in L^2(\N;E^q):\Dee_q u\in L^2(\N;E^{q+1})},\\
\Dom^q(\Lie_\T)=\set{u\in L^2(\N;E^q):\Lie_\T u\in L^2(\N;E^q)}.
\end{gather*}
With these domains they are closed, densely defined. In particular, $L^2(\B;\E^q)$ is a closed subspace of $L^2(\N;E^q)$.

Let
\begin{equation*}
\Dom^q(\D)=\set{u\in L^2(\B;\E^q):\Dee_q u\in L^2(\N;E^{q+1})}.
\end{equation*}
In view of \eqref{Commutators}, $\Dee_q u\in L^2(\N;\E^{q+1})$ if $u\in \Dom^q(\D)$ . For such $u$ we write $\D_q u$ in place  of $\Dee_q u$.

\begin{proposition}\label{domD}
If $u\in \Dom^q(\D)$ then $\D_q u\in \Dom^{q+1}(\D)$ and $\D_{q+1}\D_q u=0$. Furthermore, 
\begin{equation}\label{DWithDomain}
\D_q:\Dom^q(\D)\subset L^2(\B;\E^q) \to  L^2(\B;\E^{q+1})
\end{equation}
is closed, densely defined.
\end{proposition}

\begin{proof}
To say that $u\in \Dom^q(\D)$ is to say that $u\in \Dom^q(\Dee)$ and $\Lie_\T u=0$. Hence $\D_q u=\Dee_q u\in L^2(\N;E^{q+1})$ and $\Lie_\T\Dee_q u=0$. To complete the proof that $\D_qu\in \Dom^{q+1}(\D)$ we need to show that $\Dee_{q+1}\Dee_q u\in L^2(\N;E^{q+1})$. This we show by proving that in fact $\Dee_{q+1}\Dee_q u  = 0$ when $\Lie_\T u=0$, which in turn is a consequence of the fact that 
\begin{equation}\label{ExpectedDensity}
\text{$C^\infty(\B;\E^q)$ is dense in $L^2(\N;\E^q)$}
\end{equation}
for any $q$, as we will show in a moment. Once this assertion is proved, we will also conclude that $\Dom^q(\D)$ is dense in $L^2(\B;\E^q)$, since already $C^\infty(\B;\E^q)\subset \Dom^q(\D)$.

The ellipticity of $P_q$ implies that its spectrum consists of isolated eigenvalues of finite multiplicity. The eigenspace, $\Eigen_\lambda^q$ corresponding to a given eigenvalue $\lambda$ consists of smooth sections of $E^q$. Since $\Lie_\T$ commutes with $P_q$ (a consequence of \eqref{iTisSymmetric} and \eqref{Commutators}) each eigenspace splits into a direct sum $\Eigen_\lambda=\bigoplus \Eigen_{\lambda,\tau}^q$ of eigenspaces of $-\im \Lie_\T\big|_{\Eigen_\lambda}:\Eigen_\lambda^q\to \Eigen_\lambda^q$. Let 
\begin{equation*}
\Pi_{\lambda,\tau}^q:L^2(\N;E^q)\to L^2(\N;E^q)
\end{equation*}
be the orthogonal projection onto $\Eigen_{\lambda,\tau}^q$.  Note that $\spec(P_q)\subset \lbra 0,\infty\rpar$ since $P_q$ is non-negative.

Let $u\in L^2(\N;E^q)$. Then of course $u=\sum_{\lambda,\tau}\Pi^q_{\lambda,\tau}u$ 
with convergence in $L^2$. We will show that if in addition $\Lie_\T u=0$, then $\Pi^q_{\lambda,\tau}u=0$ if $\tau\ne 0$, so the above becomes
\begin{equation*}
u=\sum_{\lambda}\Pi^q_{\lambda,0}u,
\end{equation*}
a series whose partial sums are elements of $C^\infty(\B;\E^q)$. Let then $\lambda$ be an eigenvalue of $P_q$ and $\tau_0\ne 0$ be one of $-\im \Lie_\T\big|_{\Eigen_\lambda}$. Let $v\in \Eigen_{\lambda,\tau_0}^q$. Then $(\Lie_\T u,v)=(u,-\Lie_\T v)$ in the pairing of a distributional and a smooth section of $E^q$. Since $\Lie_\T u=0$,
\begin{equation*}
0=(u,\im \tau_0 v)=-\im\tau_0 (u,v).
\end{equation*}
Thus $(u,v)=0$ for all $v\in \Eigen_{\lambda,\tau_0}^q$, so $\Pi_{\lambda,\tau_0}u=0$. Since $\tau_0\ne 0$ and $\lambda$ are arbitrary, the Fourier series of $u$ has no terms with $\tau\ne 0$. 
\end{proof}

The closure of 
\begin{equation*}
\Dee_q:C^\infty(\N;E^q)\subset L^2(\N;E^q) \to  L^2(\N;E^{q+1})
\end{equation*}
is the operator $\Dee_q$ with its maximal domain: the operator has only one closed extension. The same holds for the formal adjoint, 
\begin{equation*}
\Dee^\star:C^\infty(\N;E^{q+1})\subset L^2(\N;E^{q+1}) \to  L^2(\N;E^q);
\end{equation*}
the domain of its closure is its maximal domain. A consequence of this is that the latter is the Hilbert space adjoint of the former (since $\N$ has no boundary). Let 
\begin{equation*}
\Dom^{q+1}(\D^\star)=\set{v\in L^2(\B;\E^{q+1}):\Dee^\star v\in L^2(\N;E^q)}.
\end{equation*}
As in Proposition \ref{domD}, if $v\in \Dom^{q+1}(\D^\star)$ then $\D^\star v\in L^2(\B;\E^q)$, and $\D_q^\star$ with its given domain is closed and densely defined. 

\begin{proposition}
The operator
\begin{equation*}
\D^\star:\Dom^{q+1}(\D^\star)\subset L^2(\B;\E^{q+1})\to L^2(\B;\E^q)
\end{equation*}
is the adjoint of \eqref{DWithDomain}.
\end{proposition}
\begin{proof}
Let $\frakD^{q+1}\subset L^2(\B;\E^{q+1})$ denote the domain of the adjoint of \eqref{DWithDomain}. If $v\in L^2(\B;\E^{q+1})$ belongs to $\Dom^{q+1}(\Dee^\star)$, then 
\begin{equation*}
(\D u,v)=(\Dee_q u,v)=(u,\Dee_q^\star v)
\end{equation*}
for any $u\in\Dom^q(\D)$ so $\Dom^q(\D)\ni u \mapsto (\D_q u,v)$ is $L^2$-continuous, hence $v\in \frakD^{q+1}$. Thus $\Dom^{q+1}(\D^\star)\subset \frakD^{q+1}$. Conversely, suppose $v\in \frakD^{q+1}$, so $(\Dee_q u,v)=(u,\Dee_q^*v)$ 
\begin{equation}\label{InAdjointDomain}
\Dom^q(\D)\ni u\mapsto (\D_q u,v)\in \C
\end{equation}
is $L^2$-continuous. In particular, $|(\D_q u,v)|\leq C\|u\|$ when $u\in C^\infty(\B;\E^q)\subset \Dom^q(\Dee)$. If $w\in C^\infty(\N;E^q)$ is orthogonal to $L^2(\B;\E^q)$, then $(\Dee_qw,v)=0$: indeed, $(\Dee_qw,v)=(w,\Dee_q^\star v)$ in the distributional sense, but $\Lie_\T\Dee_q^\star v=\Dee_q^\star \Lie_\T v=0$, so $(w,\Dee_q^\star v)=0$. Thus if $u,w$ are smooth with $\Lie_\T u=0$ and $w\perp u$, then
\begin{equation*}
|(\Dee_q(u+w),v)|=|(\Dee_qu,v)|=|(\D_qu,v)|\leq C\|u\|\leq C\|u+w\|.
\end{equation*}
Thus $v\in \Dom^{q+1}(\Dee^\star)$.
\end{proof}
The domain of the Hodge Laplacian in degree $q$ of the complex \eqref{TheGeneralComplex} is therefore
\begin{equation*}
\Dom(\Laplacian_q)=\set{u\in \Dom^q(\Dee)\cap \Dom^q(\Dee^\star): \Dee u\in \Dom^{q+1}(\Dee^\star),\ \Dee^\star u\in \Dom^{q-1}(\Dee),\ \Lie_\T u=0}.
\end{equation*}
This is the intersection of the domain of $P_q$, which is the Sobolev space $\mathrm H^2(\N;E^q)$, with $\ker\Lie_\T^2$, and $\Laplacian_q$ is the restriction of $P_q$ to $\Dom(\Laplacian_q)$.

The operator
\begin{equation*}
P_q:\mathrm H^2(\N;E^q)\subset L^2(\N;E^q)\to L^2(\N;E^q)
\end{equation*}
is elliptic  selfadjoint. Let
\begin{equation*}
\Pi^q_{0,0}:L^2(\N;E^q)\to L^2(\N;E^q)
\end{equation*}
be the orthogonal projection onto $\ker P_q=\Ha^q$, a subspace of $L^2(\B;\E^q)$, and let $G_q$ be the selfadjoint parametrix of $P_q$ such that 
\begin{equation*}
P_qG_q=\Id-\Pi^q_{0,0},\quad G_qP_q=\Id-\Pi^q_{0,0}.
\end{equation*}
Since $P_q$ commutes with $\Lie_\T$, so does $G_q$. 
Finally, let 
\begin{equation*}
\G_q:L^2(\B;\E^q)\to \Dom(\Laplacian_q)
\end{equation*}
be the restriction of $G_q$ to $L^2(\B;\E^q)=L^2(\N;\Wedge^q\Hor^*)\cap \ker \Lie_\T$. Then
\begin{equation*}
\Laplacian_q \G_q=\Id-\Pi^q_{0,0},\quad \G_q\Laplacian_q=\Id-\Pi^q_{0,0}.
\end{equation*}
Thus we have the Hodge decomposition
\begin{equation*}
L^2(\B;\E^q) = \Ha^q+\rg \D^\star+\rg \D
\end{equation*}
and the Hodge isomorphism 
\begin{equation*}
\Ha^q\to H^q(\B,\E)
\end{equation*}
with the cohomology spaces of the complex \eqref{TheGeneralComplex}.

Thus, the cohomology spaces $H^q(\B,\E)$ are finite dimensional. This concludes the proof of Proposition \ref{FiniteCohomology}.

Let $\h:H^q(\B;\E) \to \Ha^q$ denote the canonical Hodge isomorphism, let 
\begin{equation*}
\Phi^q_\h=\Pi^q_{0,0}\circ\Phi^q\big|_{\Ha^q}:\Ha^q\to\Ha^q.
\end{equation*}
If $u\in \Ha^q$, then the class of $\Phi^q(u)$ is by definition $\pmb \Phi^q$ applied to the class of $u$. In other words, $\Pi^q_{0,0}\circ \Phi^q\big(\h (\pmb u)\big)=\h\big(\pmb \Phi^q( \pmb u)\big)$ if $\pmb u\in H^q(\B;\E)$: the diagram
\begin{equation*}
\begin{tikzcd}[column sep=normal]
\Ha^q \arrow[r,"\Phi^q_\h"] & \Ha^q \\
H^q(\B ;\E) \arrow[u,"\h"'] \arrow[r,"\pmb \Phi^q"] & H^q(\B;\E) \arrow[u,"\h"']
\end{tikzcd}
\end{equation*}
commutes and therefore
\begin{equation*}
\tr \pmb \Phi^q=\tr \Phi^q_\h.
\end{equation*}

For sake of reference we state:
\begin{proposition}
The Lefschetz number of the complex \eqref{TheGeneralComplex} is given by 
\begin{equation}\label{HarmonicLefschetz}
L_{f,\Phi}=\sum_{q=0}^m (-1)^q\tr \Phi^q_\h.
\end{equation}
\end{proposition}


\section{Connections, actions, averages}\label{ConnectionsGroupsAverage}

We now prove Proposition \ref{OnConnections}. Suppose $\pi:F\to \N$ is a complex vector bundle of rank $r$ with Hermitian metric $h$ and connection $\nabla$.

Since the connection is Hermitian, the one-parameter group of diffeomorphisms generated by $\hat\T$, the horizontal vector field on $F$ projecting on $\T$, consists of unitary morphisms $\A_t:F\to F$ covering $\a_t$. 

Let $\hor$ be the horizontal subbundle of $TF$ according to the connection, let $\vert$ be the vertical subbundle, the kernel of $d\pi:TF\to T\N$. Define $\hat \gg$ to be the metric for which $TF=\hor\oplus \vert$ is an orthogonal decomposition, $d\pi:\hor_\eta \subset T_\eta F \to T_{\pi(\eta)}\N$ is an isometry for each $\eta\in F$, and $\hat\gg$ is the canonical metric determined by the real part of $h$ on $\vert$. Denote by $\gg'$ the metric on $\hor$.

\begin{proposition}\label{UnitaryIsometryGroup}
The diffeomorphisms $\A_t:F\to F$ preserve the decomposition $\hor\oplus \vert$ iff the curvature $\Omega:F\to F\otimes \Wedge^2\N$ of $\nabla$ has image in $F\otimes \Wedge^2\Hor^*$. If this is the case, then $\Lie_{\hat\T}\hat \gg=0$.
\end{proposition}

We first address Proposition \ref{OnConnections}. The proof of Proposition \ref{UnitaryIsometryGroup} is given afterwards. 

Assuming that the curvature of $\nabla$ has values in $F\otimes\Wedge^2\Hor^*$, the proposition ensures the existence of the Riemannian metric $\hat \gg$ on $F$ which is preserved by $\hat \T$. Let $SF\to \N$ be the sphere bundle of $F$. Since the connection is Hermitian, $\hat\T$ is tangential to $SF$ and so the restriction of $\A_t$ to $SF$ is an isometry of $SF$ with the restricted metric. The closure of $\set{\A_t|_{SF}:t\in \R}$ in the compact-open topology of the group of homeomorphisms of $SF$ is thus an abelian subgroup of the group of isometries of $SF$, which is compact becasue $SF$ is compact. It is an elementary fact that that each of these isometries is the restriction to $SF$ of a unique unitary morphism $F\to F$ and so the collection of all morphisms obtained this way is a compact abelian group $\hat G$ acting on $F$ by unitary morphisms. The existence of a surjective map $\wp:\hat G\to G$ such that
\begin{equation}\label{wp}
\begin{tikzcd}
F \arrow[d, "\pi"'] \arrow[r,"\A_{\hat g}"] & F \arrow[d, "\pi"]\\
\N \arrow[r,"\a_{\wp(\hat g)}"] &\N
\end{tikzcd}
\end{equation}
commutes for every $\hat g\in \hat G$ follows from density and the condition that $\wp(\A_t)=\a_t$. 

\medskip

\begin{proof}[Proof of Proposition \ref{UnitaryIsometryGroup}]
Since $t\mapsto \hat \A_t$ is a one-parameter group, we only need to prove the conclusion locally for small $t$. Let $p_0\in \N$, let $\X$ be a hypersurface containing $p_0$ with $\T$ transversal to $\X$. Assume that $\X$ is diffeomorphic to a ball in $\R^{n-1}$, it and $\delta>0$ so small that 
\begin{equation*}
\X\times (-\delta,\delta)\ni (p,t) \mapsto \a_t p\in \N
\end{equation*}
is a diffeomorphism onto its image. Then the image, denoted $\calU$, is the domain of a local chart $(x^1,\dots,x^n)$ with $\T x^j=0$ for $j<n$ and $\T x^n=1$ with $x_n=0$ on $\X$. Thus $\T=\partial/\partial x^n$ and 
\begin{equation*}
\a_t(x^1,\dots,x^{n-1},0)=(x^1,\dots,x^{n-1},t)
\end{equation*}
Next, choose an orthonormal frame for $F$ along $\X$ and extend it so that $\eta_\mu(\a_t p)=\hat\A_t\big(\eta_\mu(p)\big)$ when $p\in \X$. Thus $\eta_\mu(\a_t p)$ is parallel transport of $\eta(p)$ from $p$ to $\a_t(p)$ along the curve $s\mapsto \a_s(p)$.  

In particular, the $\eta_\mu$ form an orthonormal frame of $F$ over $\calU$ and $\nabla_{\partial/\partial x^n}\eta_\mu=0$. Let the $\omega^\nu_\mu$ be the respective connection forms,
\begin{equation*}
\nabla \eta_\mu=\sum_\nu \eta_\nu\otimes \omega^\nu_\mu.
\end{equation*}
Since the connection is Hermitian and the frame is orthonormal, $\overline \omega^\mu_\nu=-\omega^\nu_\mu$. Denoting by $(x,\zeta)$ the induced coordinates on $\pi^{-1}(\calU)$, the real vector fields
\begin{equation}\label{TheHorFrame}
X_j=\frac{\partial }{\partial x^j}-\sum_{\mu,\nu}\big\langle \omega^\nu_\mu,\frac{\partial }{\partial x^j}\big\rangle\big(\zeta^\mu\frac{\partial }{\partial \zeta^\nu}-\overline \zeta^\nu\frac{\partial }{\partial \overline \zeta^\mu}\big),\quad j=1,\dotsc,n.
\end{equation}
form a frame of $\hor$ over $\pi^{-1}(\calU)$; the metric is such that 
\begin{equation*}
\hat \gg(X_j,X_k)=g(\partial_{x^j},\partial_{x^k}).
\end{equation*}  
The vector fields
\begin{equation}\label{TheVertFrame}
\frac{\partial }{\partial \Re \zeta^\mu},\ \frac{\partial }{\partial \Im \zeta^\nu}, \quad\mu,\nu=1,\dots,r.
\end{equation}
span the part of $\vert$ over $\pi^{-1}(\calU)$. The metric is such that the norm of a tangent vector
\begin{equation*}
v=\sum_\mu \big( a^\mu\frac{\partial }{\partial \Re \zeta^\mu}+ b^\mu \frac{\partial }{\partial \Im \zeta^\mu}\big)
\end{equation*}
with real coefficients $a^\mu$, $b^\mu$ is that of the element
\begin{equation*}
\sum_\mu (a^\mu\eta_\mu+b^\mu \im \eta_\mu)=\sum(a^\mu+\im b^\mu)\eta_\mu
\end{equation*}
(obtained via the canonical identification of the tangent space of a vector space at a point with the vector space itself). Thus the metric $\hat \gg$ makes \eqref{TheVertFrame} into an orthonormal frame of the part of $\vert$ over $\pi^{-1}(\calU)$. 

Since $0=\nabla_{\partial/\partial x^n}\eta_\mu=\sum_\nu\eta_\nu\otimes \langle \omega^\nu_\mu,\partial/\partial x^n\rangle$, 
\begin{equation}\label{CurvatureFormsAreInHstar}
\inner_\T\omega^\nu_\mu=0\quad \text{for all }\nu,\mu.
\end{equation}
In particular,
\begin{equation*}
X_n=\frac{\partial }{\partial x^n}.
\end{equation*}
Thus, in coordinates, the integral curve of $X_n$ starting at $\eta(0)=\sum_\mu \zeta^\mu\eta_\mu(p)$ with $p\in \X$ of coordinates $(x^1,\dots,x^{n-1},0)$ is just
\begin{equation*}
t \mapsto \A_t(\eta(0))=\sum_\mu \zeta^\mu \eta_\mu(\a_t p)=(x^1,\dots,x^{n-1},t;\zeta_1\dots,\zeta_r)=\eta(t).
\end{equation*}
Therefore
\begin{equation*}
d\A_t\big(\frac{\partial }{\partial \Re \zeta^\mu}\big) = \frac{\partial }{\partial \Re \zeta^\mu},\quad d\A_t\big(\frac{\partial }{\partial \Im \zeta^\mu}\big) = \frac{\partial }{\partial \Im \zeta^\mu},
\end{equation*}
so $d\A_t$ always preserves $\vert$. 

On the other hand, 
\begin{equation*}
d\A_t(X_j(\eta(0)))=X_j(\eta(t))+V_j(t)
\end{equation*}
with
\begin{equation*}
V_j(t)=\sum_{\mu,\nu}\big[\big\langle \omega^\nu_\mu(\a_tp),d\a_t\big(\frac{\partial }{\partial x^j}\Big|_p\big)\big\rangle-\big\langle \omega^\nu_\mu(p),\frac{\partial }{\partial x^j}\Big|_p\big\rangle\big]\big(\zeta^\nu\frac{\partial }{\partial \zeta^\mu}-\overline \zeta^\mu\frac{\partial }{\partial \overline \zeta^\nu}\big),
\end{equation*}
an element of $\vert_{\eta(t)}$. Thus $d\A_t$ preserves $\hor$ iff 
\begin{equation}\label{omegaConstant}
\big\langle \omega^\nu_\mu(\a_tp),d\a_t\big(\frac{\partial }{\partial x^j}\big)\big\rangle=\big\langle \omega^\nu_\mu(p),\frac{\partial }{\partial x^j}\big\rangle\quad\text{for all $j$ and all small $t$,}
\end{equation}
in other words, $\a_t^*(\omega^\nu_\mu(\a_tp))$ is constant. 

Hence, if $\A_t$ preserves the decomposition $\hor\oplus\vert$, then $\Lie_\T\omega^\nu_\mu=0$. Together with \eqref{CurvatureFormsAreInHstar} this yields
\begin{equation*}
\inner_\T d\omega^\nu_\mu=0\quad \text{for all }\nu,\mu
\end{equation*}
since $\Lie_\T\omega^\nu_\mu = \inner_\T d\omega^\nu_\mu + d(\inner_\T \omega^\nu_\mu)$. Applying this and \eqref{CurvatureFormsAreInHstar} to the curvature matrix 
\begin{equation*}
[\Omega^\nu_\mu]=[d\omega^\nu_\mu]+[\sum_\lambda \omega^\nu_\lambda\wedge \omega^\lambda_\mu]
\end{equation*}
we get $\inner_\T\Omega=0$ if the decomposition $\hor\oplus\vert$ is preserved.

On the other hand, if $\inner_\T\Omega=0$, then $\inner_\T d\omega^\nu_\mu=0$ since $\inner_\T\omega^\nu_\mu=0$ anyway. Then \eqref{omegaConstant} holds, so $\A_t$ preserves the decomposition $\hor\oplus\vert$. 

Finally, if $\inner_\T\Omega=0$, then $\A_t$ preserves the horizontal frame \eqref{TheHorFrame} and the vertical frame \eqref{TheVertFrame}. In particular, from the definition of the metric on $\hor$ we do have
\begin{equation*}
\gg'(d\A_t X_j,d\A_t X_k)=\gg(d\a_t \frac{\partial }{\partial x^j},d\a_t \frac{\partial }{\partial x^k})=\gg(\frac{\partial }{\partial x^j},\frac{\partial }{\partial x^k})=\gg'(X_j,X_k),
\end{equation*}
while from the definition of the metric on $\vert$ we also get $\A_t^*\gg'=\gg'$ along this summand. This completes the proof of Proposition \ref{OnConnections}.
\end{proof}

For a discussion on the existence of connections with $\Wedge^2\Hor^*$-valued curvature, see Appendix~\ref{CurvatureAndGroups}.

\smallskip
Continuing with a connection on $F$ whose curvature is valued in $F\otimes \Wedge^2\Hor^*$, with $\Lie_\T$ defined as $\nabla_\T$, and $\A_t$ as the one-parameter group of unitary morphisms defined by parallel transport along the integral curves of $\T$.

The operator $\Av:C^\infty(\N;F)\to C^\infty(\N;F)$ defined as that in \eqref{defOfAverage} has a continuous extension to $L^2(\N;F)$. Let $C^\infty(\B;\F)=\set{u\in C^\infty(\N;F):\Lie_\T u=0}$, let $L^2(\B;\F)$ denote the kernel of $\Lie_\T$ in $L^2(\N;F)$. 

\begin{lemma}\label{AvIsProjector}
The operator $\Av:L^2(\N;F)\to L^2(\N;F)$ is the orthogonal projection on $L^2(\B;\F)\subset L^2(\N;F)$.
\end{lemma}

The proof consist of observing that $\Av$ is a selfadjoint projection. 

\begin{example}\label{ExampleOfAction}
Let $\N=S^5\subset \C^3$ be the unit sphere. With the standard coordinates $(z^1,z^2,z^3)$ on $\C^3$, let 
\begin{equation*}
\T=\im\big(\tau(z_1\partial_{z_1}-\overline z_1\partial_{\overline z_1})+(z_2\partial_{z_2}-\overline z_2\partial_{\overline z_2})+2(z_3\partial_{z_3}-\overline z_3\partial_{\overline z_3})\big)
\end{equation*}
with $0\ne \tau\in \R$. This vector field is real, tangent to $S^5$, and preserves the standard metric on $S^5$. The one-parameter group of diffeomorphisms generated by $\T$ is 
\begin{equation*}
\a_t(z)=(e^{\im \tau t}z_1,e^{\im t}z_2,e^{2\im  t}z_3),
\end{equation*}
Assuming $\tau/2\pi$ is irrational, the closure of $\set{\a_t:t\in\R}$ in the compact-open topology is the $2$ torus $G=S^1\times S^1$, with 
\begin{equation*}
\a_\omega(z)=(\omega_1z_2,\omega_2 z_2,\omega_2^2z_3).
\end{equation*}
Let $F=S^5\times \C^3$ as a trivial Hermitian vector bundle over $S^5$. Denoting by $(\zeta_1,\zeta_2,\zeta_3)$ the coordinates on $\C^3$ with respect to the standard basis, let the $\sigma_i\in \R$ be positive, define
\begin{equation*}
\hat \T= \T+\im\big(\sigma_1(\zeta_1\partial_{\zeta_1}-\overline \zeta_1\partial_{\overline \zeta_1})+\sigma_2(\zeta_2\partial_{\zeta_2}-\overline \zeta_2\partial_{\overline \zeta_2})+\sigma_3(\zeta_3\partial_{\zeta_3}-\overline \zeta_3\partial_{\overline \zeta_3})\big).
\end{equation*}
Again this vector field preserves the product Riemannian metric on $F$. The one-parameter group generated by $\hat \T$ is given by the unitary maps
\begin{equation*}
\A_t(z;\zeta)=(\a_t(z);e^{\im \sigma_1 t}\zeta_1,e^{\im \sigma_2 t}\zeta_2,e^{\im \sigma_3 t}\zeta_3).
\end{equation*}
Assuming that the $\sigma_i/2\pi$ are independent over $\Q$, the closure is 
\begin{equation*}
\hat G=(S^1)^2\times (S^1)^3
\end{equation*}
with 
\begin{equation*}
\A_{(\omega,\theta)}(z;\zeta)=(\omega_1z_2,\omega_2 z_2,\omega_2^2z_3;\theta_1\zeta_1,\theta_2\zeta_2,\theta_3\zeta_3),\quad (\omega,\theta)\in \hat G,\ (z;\zeta)\in F
\end{equation*}
and $\wp(\omega;\theta)=\omega$. The isotropy group of $G$ at a point $z^0=(0,0,z_3^0)\in S^5$ is 
\begin{equation*}
\isotropy = (S^1\times\set{1})\cup (S^1\times\set{-1})
\end{equation*}
(which is not connected). 

The following illustrates a situation that will come up in Section \ref{sLimit}. The preimage of $\isotropy$ by $\wp$ is 
\begin{equation*}
\preisotropy = (S^1\times\set{1}\times (S^1)^3) \cup (S^1\times\set{-1}\times (S^1)^3),
\end{equation*}
The subgroup $\hat G_0=\set{1}\times S^1\times (\set{1})^3\subset \hat G$ is transverse to $\preisotropy$ at the identity element, a finitely sheeted covering of (actually diffeomorphic to) $G_0=\set{1}\times S^1$, the latter also a double cover of the $G$-orbit $\set{(0,0,z_3)}\subset S^5$ via $G_0\ni \omega\mapsto a_\omega (z^0)$. Any $1$-dimensional subgroup $\hat G_0'\subset \hat G$ with transversal intersection to $\preisotropy$ has $\#(\hat G_0'\cap \preisotropy)=2$.

\end{example}


\section{Passage to $L^2$ traces}\label{L2Traces}

The approach to the Atiyah-Bott formula for the Lefschetz number using the heat kernels as mollifiers following Kotake \cite{Kotake} works here as well. Postponing the specifics to Section \ref{Limits}, suppose 
\begin{equation}\label{CommuteAndSmoothing}
\display{280pt}{
the operators $\Phi^q_s:L^2(\N;E^q)\to L^2(\N;E^q)$, $s>0$, commute with $\Lie_\T$ and have smooth Schwartz kernel.}
\end{equation}
Then, in particular, they map $L^2(\B;\E^q)$ to itself. We do not assert that $\Phi^q$ is obtained as a limit of the $\Phi^q_s$ as $s\to 0^+$.

\begin{lemma}
The operators 
\begin{equation*}
\Phi^q_s\big|_{L^2(\B;\E^q)}:L^2(\B;\E^q)\to L^2(\B;\E^q)
\end{equation*}
are trace class.
\end{lemma}

Indeed, let $\set{\psi^0_k}$ be an orthonormal basis of $L^2(\B;\E^q)$, let $\set{\psi^\perp_k}$ be one of $L^2(\B;\E^q)^\perp$, together denoted $\set{\psi_k}$. Then the absolute convergence of 
\begin{equation*}
\Tr (\Phi^q_s) =\sum_k (\Phi^q_s\psi_k,\psi_k)
\end{equation*}
implies that of 
\begin{equation*}
\Tr (\Phi^q_s\big|_{L^2(\B;\E^q)}) =\sum_k (\Phi^q_s\psi^0_k,\psi^0_k).
\end{equation*}

Assume further that 
\begin{equation}\label{sGeometricEndo}
\D_q\circ \Phi^q_s=\Phi^{q+1}_s\circ \D_q,
\end{equation}
so they also give a geometric endomorphism. Followed by projections they then give maps $\Ha^q\to \Ha^q$; assume
\begin{equation}\label{PhiAt0}
\Phi^q_s = \Phi^q\text{ on }\Ha^q.
\end{equation}

\begin{proposition}
If \eqref{CommuteAndSmoothing}, \eqref{sGeometricEndo}, and \eqref{PhiAt0} hold, then Lefschetz number of the complex \eqref{TheGeneralComplex} is given by
\begin{equation}\label{LefshetzAsL2trace}
L_{f,\Phi}=\sum_{q=0}^m (-1)^q \Tr (\Phi^q_s\big|_{L^2(\B;\E^q)}).
\end{equation}
\end{proposition}

The proof consists of a sequence of well known observations. Let 
\begin{equation*}
\Pi^q_{\lambda,0}:L^2(\B;\E^q)\to L^2(\B;\E^q)
\end{equation*}
be the orthogonal projection on $\Eigen^q_{\lambda,0}\subset L^2(\B;\E^q)$, the kernel of $\Lie_\T$ in the eigenspace  of $P_q$ corresponding to the eigenvalue $\lambda$. One then has
\begin{equation*}
\Tr (\Phi^q_s\big|_{L^2(\B;\E^q)})=\sum _\lambda \Tr (\Pi^q_{\lambda,0}\circ \Phi^q_s\big|_{\Eigen^q_{\lambda,0}}).
\end{equation*}
It is easy to see that $\D_q$ maps $\Eigen^q_{\lambda,0}$ into $\Eigen^{q+1}_{\lambda,0}$ for each $\lambda$, so one has a complex
\begin{equation}\label{EigenDeRham}
\cdots \to \Eigen^{q-1}_{\lambda,0} \xrightarrow{\D_{q-1}} \Eigen^q_{\lambda,0} \xrightarrow{\D_q} \Eigen^{q+1}_{\lambda,0}\to \cdots
\end{equation}
of finite-dimensional spaces. But the complex is acyclic if $\lambda\ne 0$, so 
\begin{equation*}
\sum_{q=0}^m (-1)^q \Tr (\Phi^q_s\big|_{L^2(\B;\E^q)})=\sum_{q=0}^m (-1)^q \Tr (\Pi^q_{0,0}\circ \Phi^q_s\big|_{\Eigen^q_{0,0}}).
\end{equation*}
Observing now that $\Eigen^q_{0,0}$ is $\Ha^q$, the harmonic space of the complex \eqref{TheGeneralComplex} in degree $q$, one gets from \eqref{PhiAt0} that \eqref{LefshetzAsL2trace} is \eqref{HarmonicLefschetz}. 

\medskip
As is well known the trace of  $\Phi^q_s:L^2(\N;E^q)\to L^2(\N;E^q)$ is given by an integral along the diagonal in $\N\times\N$. With the correct interpretation, also the trace of $\Phi^q_s\big|_{L^2(\B;\E^q)}$ can be written as an integral over the diagonal of $\B\times \B$. This however makes the analysis much harder if $\B$ is not a smooth manifold. Alternatively, we use:

\begin{proposition}
The operator $\Phi^q_s\circ \Av^q:L^2(\N;E^q)\to L^2(\N;E^q)$ is trace class. Furthermore
\begin{equation*}
\Tr(\Phi^q_s\big|_{L^2(\B;\E^q)})=\Tr(\Phi^q_s\circ\Av^q).
\end{equation*}
\end{proposition}

The proof is elementary. The first assertion follows from the fact that $\Av^q$ is a bounded operator and that trace class operators form a two-sided ideal. For the second we use that $\Av^q$ is the orthogonal projection on $L^2(\B;\E^q)$, the fact that $\Phi^q_s$ preserved $L^2(\N;E^q)$, and the basis above:
\begin{multline*}
\Tr (\Phi^q_s\circ \Av^q)=\sum_k \big((\Phi^q_s\circ \Av^q)\psi^0_k,\psi^0_k\big)+\sum_k \big((\Phi^q_s\circ \Av^q) \psi^\perp_k,\psi^\perp_k\big)\\=\sum_k \big((\Phi^q_s\circ \Av^q) \psi^0_k,\psi^0_k\big)=\Tr(\Phi^q_s\big|_{L^2(\B;\E^q)})
\end{multline*}
using $\Av^q\psi^\perp_k=0$ and $\Av^q\psi^0_k=\psi^0_k$.

\medskip
This proposition is the key step. Now the rest of the argument can take advantage of classical distribution theory. Below and later sections we write $K_A$ for the Schwartz kernel of the operator $A$. 

\medskip

The trace of $\Phi^q_s\circ \Av^q$ is of course
\begin{equation*}
\Tr (\Phi^q_s\circ \Av^q) = \int_\N \tr K_{\Phi^q_s\circ \Av^q}(p,p)\,d\m(p).
\end{equation*}
It is convenient to rewrite this as a distributional pairing. Let $\Delta\subset \N\times \N$ be the diagonal, $\iota_\Delta:\N\to \N\times\N$ the diagonal map. The trace of a linear map $S:V\to V$ on a finite dimensional vector space is given as the pairing of its tensor representation, $S_\tensor \in V\otimes V^*$, with the tensor representation, $\calI_\tensor\in V^*\otimes V$, of the identity map $V^*\to V^*$:
\begin{equation*}
\tr S=\langle S_\tensor,\calI_\tensor\rangle.
\end{equation*}
Using this with the section $\calI_\tensor^q\in C^\infty(\N;E^{q,*}\otimes E^q)$  representing the identity map $E^{q,*}\to E^{q,*}$ we get
\begin{equation}\label{BeforeLimit}
\Tr (\Phi^q_s\circ \Av^q) = \int_\N \big\langle K_{\Phi^q_s\circ \Av^q}(p,p), \calI_\tensor^q(p)\big\rangle\,d\m(p)
=\big\langle \iota_\Delta^* K_{\Phi^q_s\circ \Av^q}, \calI_\tensor^q\big\rangle.
\end{equation}

Under the right conditions and an appropriate choice of $\Phi^q_s$ we will have, in Section \ref{Limits}, that $\iota^*_\Delta(K_{\Phi^q_s\circ \Av^q})$ converges as a distribution as $s\to 0^+$, for each $q$. Their pairings with the smooth sections $\calI_\tensor^q$ in the various degrees will give the formula in the style of Atiyah-Bott for the right hand side of the Lefschetz number.


\section{Wave front sets}\label{WaveFront}

Complexes will not be relevant in the next few sections, so we revert to the less specific setting of Section~\ref{ConnectionsGroupsAverage}. 

We are given a Hermitian vector bundle $F\to\N$ with Hermitian connection $\nabla$ whose curvature is a smooth section of $F\otimes \Wedge^2\Hor^*$. The group $\hat G$ is the closure of the one-parameter group of unitary bundle maps $F\to F$ generated by $\Lie_\T=\nabla_\T$, $\wp:\hat G\to G$ is the canonical homomorphism, the action $\A_{\hat g}:F\to F$ covers $\a_{\wp(\hat g)}$. We will be using the averaging operator $\Av:C^\infty(\N;F)\to C^\infty(\N;F)$. 

We are given a bundle homomorphism $\phi:f^*F\to F$ for which the analogue of \eqref{CDforPhi} holds, and define $\Phi:C^\infty(\N;F)\to C^\infty(\N;F)$ in analogy with \eqref{DefOfPhiq}. Let $K_{\Phi}$ be the Schwartz kernel of this operator. We wish to determine its wave front set, as well as that of $K_{\Phi\circ \Av}$. 

The Schwartz kernel of $f^*:C^\infty(\N)\to C^\infty(\N)$ is the $\delta$ along the graph $\Gamma$ of $f$, the distribution $K_{f^*}$ defined by 
\begin{equation*}
\langle K_{f^*},w\rangle =\int_\N \iota_\Gamma^*w\,d\m,\quad w\in C^\infty(\N\times \N),
\end{equation*}
with $\iota_\Gamma:\N\to \N\times \N$ given by $\iota_\Gamma(p)=(p,f(p))$. We will write a formula for $K_\Phi$ using $K_{f^*}$ which will give $\WF(K_{\Phi})=\WF(K_{f*})$.

The vector smooth bundle homomorphism $\phi:f^*F\to F$ can be viewed as a section of the vector bundle $F\boxtensor F^*\to \N\times \N$ along $\Gamma$. At $p\in \N$, the map $\phi(p):(f^*F)_p\to F_p$ sends an element of $F_{f(p)}$ to an element of $F_p$. Writing $\phi_\tensor(p,f(p))$ for the tensor representation of $\phi(p)$ in $F_p\otimes F^*_{f(p)}$ we get the claimed section
\begin{equation}\label{phiAsTensor}
(p,f(p))\mapsto \phi_\tensor(p,f(p)) 
\end{equation}
of $F\boxtensor F^*$ along $\Gamma$. Since $\Gamma$ is closed, there is an element in $C^\infty(\N \times \N;F\boxtensor F^*)$, which we continue to denote $\phi_\tensor$, whose restriction to $\Gamma$ is the section referred to above (Whitney \cite{Whitney1935}). 

We will now modify $\phi_\tensor$ in the complement of $\Gamma$ while keeping it smooth. Let $\A_{\hat g^{-1}}^*:F^*\to F^*$ be the homomorphism (covering $\a_{{\wp(\hat g)}}$) dual to  $\A_{\hat g^{-1}}:F\to F$. The commutativity of the diagram \eqref{CDforPhi} is reflected in the property
\begin{equation}\label{Requivariance}
\A_{\hat g}\boxtensor \A_{\hat g^{-1}}^* \big(\phi_\tensor (p,p')\big)=\phi_\tensor (\a_{\wp(\hat g)}(p),\a_{\wp(\hat g)}(p'))
\end{equation}
of $\phi_\tensor$ when $(p,p')\in \Gamma$ (i.e., when $p'=f(p)$). The section 
\begin{equation*}
\tilde \phi_\tensor(p,p')=\int_{\hat G} \A_{\hat g^{-1}}\boxtimes \A_{\hat g}^* \big(\phi_\tensor (\a_{\wp(\hat g)}(p),\a_{\wp(\hat g)}(p'))\big)\,d\hat \mu(\hat g)
\end{equation*}
satisfies \eqref{Requivariance} for all $(p,p')$ and agrees with the definition of $\phi_\tensor$ along $\Gamma$. Since averaging preserves smoothness we still have that $\tilde \phi_\tensor$ is smooth. So we may assume \eqref{Requivariance} already holds for $\phi_\tensor$ for any $(p,p')$.

The pointwise pairing of $\phi_\tensor$ and a smooth section $w$ of $F^*\boxtimes F$ gives a smooth complex valued function via 
\begin{equation*}
(p,p')\mapsto \langle \phi_\tensor,\psi\rangle_{(p,p')} =  \langle \phi_\tensor(p,p'),w(p,p')\rangle
\end{equation*}
which can be paired with $K_{f^*}$:
\begin{equation*}
\big\langle K_{f^*},\langle \phi_\tensor,w\rangle\big\rangle =\int_\N \iota_\Gamma^*\langle \phi_\tensor,w\rangle\,d\m =\int_\N\langle \phi_\tensor(p,f(p)),w(p,f(p))\rangle\,d\m(p).
\end{equation*}
We will write $\phi_{\tensor}K_{f^*}$ for this distribution:
\begin{equation*}
\langle  \phi_\tensor K_{f^*},w\rangle = \big\langle K_{f^*},\langle \phi_\tensor,w\rangle\big\rangle.
\end{equation*}

If $w=v\boxtimes u\in C^\infty(\N\times \N;F^*\boxtimes F)$, then 
\begin{equation*}
\langle\phi_\tensor(p,f(p)),w(p,f(p))\rangle = \langle\phi\big(p,u(f(p))\big),v(p)\rangle=\langle\Phi(u)(p),v(p)\rangle
\end{equation*}
and the above reduces to 
\begin{equation*}
\int_\N\langle \Phi(u)(p),v(p)\rangle\,d\m(p)=\langle \Phi(u),v\rangle.
\end{equation*}
Thus 
\begin{equation*}
K_{\Phi}=\phi_\tensor K_{f^*}
\end{equation*}
is the Schwartz kernel of $\Phi:C^\infty(\N;F)\to C^\infty(\N;F)$. 

In particular:

\begin{lemma}\label{WFfStar}
The wave front set of $K_\Phi$ is
\begin{equation*}
\WF(K_\Phi)=\set{\pmb \xi\oplus \pmb \xi'\in T^*_{(p,p')}(\N \times \N)\minus 0:p,p'\in \N,\ p'=f(p),\ f^*\pmb \xi'=-\pmb \xi}.
\end{equation*}
\end{lemma}

We will also need the wave front sets of $K_{\Av}$ and $K_{\Phi\circ \Av}$. We make use of the definition of $\Av$ given in \eqref{Av} in terms of the maps in \eqref{CDforAv}. Both wave front sets are obtained as the conormal bundle of the graph of a function. First, we have 
\begin{equation*}
\WF'(K_{\hat \pi^*})=\set{\big((\pmb\xi\oplus 0)\oplus \pmb \xi''\big)\in T^*_{((p,g),p'')}\big((\N\times \hat G)\times \N\big)\minus 0:p''=p,\ \pmb \xi''=\pmb \xi}
\end{equation*}
using the notation of \cite[p. 165]{Hor71} which here means 
\begin{equation*}
\big((\pmb\xi\oplus 0)\oplus \pmb \xi''\big)\in \WF'(K_{\hat \pi^*}) \iff (\pmb\xi\oplus 0)\oplus (-\pmb \xi'')\in \WF(K_{\hat \pi^*}).
\end{equation*}
Next, we find $\WF'(K_{\hat \rho_*})$ taking advantage of the relation between the wave front set of the Schwartz kernel of an operator and that of its dual. Thus we first need to get a handle on $\WF(\rho^*)$, $\hat \rho^*:T^*\N\to T^*(\N\times \hat G)$. If $v\oplus w\in T_{(p,\hat g)}(\N\times \hat G)$, then 
\begin{equation*}
d\hat \rho(v\oplus w)=d\a_{\wp(\hat g)}(v)+\hat \kappa_{p,\hat g}(w)
\end{equation*}
for a certain linear map $\hat \kappa_{p,\hat g}:T_{\hat g}\hat G\to T_{\a_{\wp(\hat g)}(p)}\N$. With this notation, $\hat \rho^*$ is given by
\begin{equation*}
T^*_{\a_{\wp(\hat g)}(p)}\N\ni \pmb \xi \mapsto \hat \rho^*(\pmb \xi)=\a_{\wp(\hat g)}^*(\pmb \xi) \oplus \hat \kappa_{p,\hat g}^*(\pmb \xi)\in T_{(p,\hat g)}^*(\N\times \hat G),
\end{equation*}
therefore
\begin{multline*}
\WF'(K_{\hat \rho_*})=\{\pmb \xi\oplus (\pmb \xi'\oplus \pmb \nu')\in T^*_{(p,(\a_{\wp(\hat g^{-1})}(p),\hat g))}\big(\N\times (\N\times \hat G)\big): \\
p\in \N,\ \hat g\in \hat G,\ \pmb \xi'\oplus\pmb \nu'=\a_{\wp(\hat g)}^*(\pmb \xi) \oplus \hat \kappa_{\a_{\wp(\hat g^{-1})}(p),\hat g}^*(\pmb \xi)\}.
\end{multline*}
Consequently $\WF'(\Av)$, coveniently computed as the composition of two relations, consists of those $\pmb \xi\oplus \pmb \xi''\in T^*(\N\times\N)$ for which there is $\pmb \xi'\oplus \pmb \nu'\in T^*(\N\times \hat G)$ such that 
\begin{equation*}
(\pmb \xi'\oplus \pmb \nu')\oplus \pmb \xi''\in \WF'(K_{\hat \pi^*}) \text{ and } \pmb \xi\oplus (\pmb \xi'\oplus \pmb \nu')\in \WF'(K_{\hat \rho_*}).
\end{equation*}
If $\pmb \xi\in T^*_p\N$, $(\pmb \xi'\oplus \pmb \nu')\in T_{(p',\hat g)}^*(\N\times \hat G)$, and $\pmb \xi''\in T^*_{p''}\N$, the first condition implies
\begin{equation*}
\pmb \nu'=0,\ \pmb \xi'=\pmb \xi'', \text{ and }p'=p''
\end{equation*}
which together with the second gives
\begin{equation*}
p''=\a_{\wp(\hat g^{-1})}(p),\ \pmb \xi''=\a^*_{\wp(\hat g)}(\pmb \xi), \text{ and }\hat \kappa_{\a_{\wp(\hat g^{-1})}(p),\hat g}^*(\pmb \xi)=0.
\end{equation*}
Momentarily reverting to wave front set this gives
\begin{multline}\label{WFofAV}
\WF(\Av)=\{\pmb \xi\oplus \pmb \xi''\in T^*_{(p,p'')}(\N\times\N)\minus 0: p''=\a_{\wp(\hat g^{-1})}(p)\text{ for some }\hat g\in \hat G,\\
\hat \kappa_{\a_{\wp(\hat g^{-1})}(p),\hat g}^*(\pmb \xi)=0, \text{ and }\pmb \xi''=-\a^*_{\wp(\hat g)}(\pmb \xi)\}.\phantom{quad}
\end{multline}

In place of the map $\hat \kappa_{p,\hat g}$ we can use the similarly defined map determined by $\N\times G\to \N$, $\rho(p,g)=\a_g(p)$, because $\wp :\hat G\to G$ is surjective with the consequence that $\wp^*$ is injective. 
\begin{lemma}
Let $p\in \N$, let $\Y\subset \N$ be the $G$-orbit through some point of $\N$, and let $N^*\Y\subset T^*\N$ be the conormal bundle of $\Y$. For any $p\in \Y$ and $g\in G$:
\begin{enumerate}[\ $(i)$]
\item \label{Li} the kernel of $\kappa_{\a_{g^{-1}}(p),g}^*:T^*_p\N\to T^*_g G$ is $N^*_p\Y$;
\item \label{Lii} $\a_{g^{-1}}^*$ maps $N^*_p\Y$ onto $N^*_{\a_g(p)}\Y$ .
\end{enumerate}
\end{lemma}

\begin{proof}
For fixed $p\in \N$, the map $G \ni g \mapsto \a_{g^{-1}}(p)\in \N$ is a submersion onto $\Y$ with derivative 
\begin{equation*}
\kappa_{\a_{g^{-1}}(p),g}:T_g G\to T_p\Y\subset T_p\N.
\end{equation*}
Thus the annihilator of the image of $\kappa_{\a_{g^{-1}}(p),g}$, the kernel of the dual map, is the conormal fiber to $\Y$ at $p$. This proves \eqref{Li}.

For \eqref{Lii} we just note that if $g\in G$ is fixed then $d\a_{g^{-1}}$ gives an isomorphism $T_{\a_{g}(p)}\Y \to T_p\Y$. 
\end{proof}

Keeping the notation of the lemma we have that  the part of $\WF(\Av)$ over $\Y\times \Y$ is
\begin{equation*}
\set{\pmb \xi\oplus\pmb \xi''\in N^*(\Y\times \Y)\minus 0: \pmb \xi''=-\a_g^*(\pmb \xi) \text{ for some }g\in G}
\end{equation*}
and $\Av$ has no wave front set over $\Y\times \Y'$ if the $G$-orbit $\Y'$ is disjoint from $\Y$. Note also that the fiber of the conormal bundle of $\Y \times \Y\subset \N\times \N$ at $(p,p')$ is $N^*_p\Y\oplus N^*_{p'}\Y$.

\medskip
Making now use of Lemma \ref{WFfStar} we get
\begin{multline}\label{WFofPhiAv}
\phantom{\quad}\WF(K_{\Phi\circ\Av})=\{\pmb \xi\oplus \pmb \xi''\in N^*\big(\Y \times f(\Y)\big)\minus 0:  \Y\text{ is a $G$-orbit}\\\text{ and }\pmb \xi=-(f\circ \a_g)^*(\pmb \xi'')\text{ for some }g\in G\}.\quad
\end{multline}
Note if $\pmb \xi\oplus \pmb \xi''\in \WF(K_{\Phi\circ \Av})$ then $\pmb \xi$ and $\pmb \xi''$ are conormal to the same orbit. Observe also that we may have $\pmb \xi=0$ with $\pmb \xi''\ne 0$.


\section{Transversality}\label{discreteness}

We continue with the general setup of the previous section. The condition that $\WF(K_{\Phi\circ \Av}) \cap \dot N^*(\iota_\Delta)$ is empty implies  that $\iota_\Delta^*K_{\Phi\circ\Av}$ is defined in a natural way (H\"ormander \cite{Hor71}); $\iota_\Delta:\N\to \N\times\N$ the diagonal embedding. We will now derive a necessary and sufficient condition for this condition to hold (generalizing $\Delta\trans \Gamma$) with the ultimate aim to take the limit in \eqref{BeforeLimit} as $s\to 0^+$ with a simpler regularization of $K_{\Phi^q\circ \Av}$. 

\medskip
In view of \eqref{WFofPhiAv}, a covector $\pmb \xi\oplus \pmb \xi''\in T^*_{(p,p'')}\N\times\N$ belongs to $\WF(K_{\Phi\circ \Av}) \cap \dot N^*(\iota_\Delta)$ if and only if 
\begin{equation*}
\display{300pt}{
$\pmb \xi$ and $\pmb \xi''$ are conormal to the same $G$-orbit $\Y$, there is  $g\in G$ such that $\pmb \xi=-(f\circ \a_g)^*\pmb \xi''$, and $\pmb \xi=-\pmb\xi''$,
}
\end{equation*}
which in particular implies that $f(\Y)\subset \Y$ and
\begin{equation*}
\pmb \xi=(f\circ \a_g)^*\pmb \xi.
\end{equation*}
Let 
\begin{equation*}
\Fix(\pmb f)=\set{\Y:\Y\text{ is a $G$-orbit and }f(\Y)= \Y}
\end{equation*}
If $\Y\in \Fix(\pmb f)$ and $p_0\in \Y$, then $f(p_0)=\a_{g_0^{-1}}(p_0)$ for some $g_0$. Then $\a_{g_0}\circ f$ has $p_0$ as a fixed point, and by equivariance, all points of $\Y$ are fixed by $\a_{g_0}\circ f$. Therefore:

\begin{proposition}\label{WFvsDet}
The condition 
\begin{equation}\label{WFCondition}
\WF(K_{\Phi\circ \Av}) \cap \dot N^*(\iota_\Delta)=\emptyset
\end{equation}
holds if and only if for every $\Y\in \Fix(\pmb f)$, if $g_0 \in G$ is such that $\a_{g_0}\circ  f$ is the identity on $\Y$, then the equation
\begin{equation}\label{DetCondition}
\big(\Id-(\a_{g_0}\circ  f)^*\big)\pmb \xi=0,\quad \pmb \xi \in N^*\Y
\end{equation}
has no nonzero solution. The element $g_0$ is determined only modulo the isotropy group of $p_0$ in $G$.
\end{proposition}

The support of $K_{\Phi\circ \Av}$ is 
\begin{equation*}
\bigcup_\Y \Y\times f(\Y),
\end{equation*}
the preimage of the graph of $\pmb f:\B\to \B$ by the canonical  map $\N\times \N\to \B\times \B$. Indeed, to see that $K_{\Phi\circ \Av}$ vanishes in the complement of this set, consider two open sets $\calU$, $\calV\subset \N$ each invariant under $G$ with $f^{-1}(\calU)\cap \calV=\emptyset$. If $u\in C^\infty(\B;\F)$ has $\supp(u)\subset \calU$, then $\supp (\Phi\circ\Av)(u)\subset f^{-1}(\calU)$, therefore if $v\in C^\infty(\B;\F^*)$ has $\supp v\subset \calV$, then $\langle v,(\Phi\circ\Av)(u)\rangle=0$. Thus $K_{\Phi\circ\Av}$ vanishes on $\calV\times \calU$. As expected, the intersection of $\supp(K_{\Phi\circ \Av})$ and the image of $\iota_\Delta$ is 
\begin{equation*}
\set{(p,p')\in \N\times \N:\text{ there is  $g\in G$ such that $\a_g(p)=p'=f(p)$}}.
\end{equation*}
Thus
\begin{equation*}
\supp( K_{\Phi\circ \Av})\cap \Delta=\bigcup_{\Y\in \Fix(\pmb f)} \Y\times\Y.
\end{equation*}

\begin{theorem}\label{DiscreteFixedPoints}
If \eqref{WFCondition} holds, then $\Fix(\pmb f)$ is a discrete subset of $\B$.
\end{theorem}

The following set up will be used in the proof of the theorem and again in Section~\ref{sLimit}. 

Let $\Y\in \Fix(\pmb f)$, pick some $p_0\in \Y$ and let $r>0$ be such that the geodesic ball in $\N$ with center $p_0$ and radius $2r$ is strictly geodesically convex. Let $N_{p_0}\Y$ be the subspace of $T_{p_0}\N$ orthogonal to $T_{p_0}\Y$, let $B\subset N_{p_0}\Y$ be the ball with center $0$ and radius $r$. Finally let $\exp_{p_0}:T_{p_0}\N\to \N$ be the exponential map and let $\X$ the image of $B$. Thus $\X$ is transverse to $\Y$ at $p_0$. We assume $r$ so small that $\Y\cap \X=\set{p_0}$. 

Let $\isotropy$ be the isotropy group of $p_0$, let $G_0$ be a compact connected subgroup of $G$ complementary to $\isotropy$, and let $D\subset G_0$ be a neighborhood of the identity element, a ball in some invariant metric, such that $\X\times D\ni (p,g)\mapsto \a_g(p)\in \N$ is a diffeomorphism onto its image. This image is an open set $\calU\subset \N$ foliated by copies of $D$. Let $\proj:\calU\to \X$ denote the canonical projection. If $\Y'$ is an orbit of $G$ then $\proj(\Y'\cap \calU) = \Y'\cap \X$. Slices such as $\X$ are invariant under the action of $\isotropy$, and moreover, if $p\in \X$ and $g\in G$ is such that $\a_g(p)\in \X$, then $g\in \isotropy$ (see \cite[~p.~64~ff.]{AlexandrinoBettiol2015}). Since $G_0$ is compact and transverse to $\isotropy$ at the identity, $G_0\cap \isotropy$ is finite, all intersections being transversal. 

There is $g_0$ such that $\a_{g_0}(f(p_0))=p_0$. With such $g_0$, the function $f_0=\a_{g_0}\circ f$ leaves $\Y$ pointwise fixed. Then we can find a positive number $r'<r$ such that $f_0(p)\in \calU$ if $p\in \X'=\exp_{p_0}(B')$ where $B'\subset N_{p_0}\Y$ is the ball of radius $r'$ centered at $0$.

\begin{proof}[Proof of Theorem \ref{DiscreteFixedPoints}]
We argue in the context of the preceding paragraphs, with the specified elements $\Y\in \Fix(\pmb f)$ and $p_0\in \Y$; without loss of generality assume already that $f$ itself leaves $\Y$ pointwise fixed. Suppose there is a sequence $\set{\Y_j}_{j=1}^\infty\subset \Fix(\pmb f)$ of distinct elements converging to $\Y$ in $\B$. We will contradict \eqref{WFCondition} by showing that the equation \eqref{DetCondition} has  nontrivial solutions.

Passing to a subsequence of $\set{\Y_j}$ we may assume that $\dist(\Y_j,\Y)<r'$  for each $j$ and then pick $p_j\in \Y_j\cap \X$; this is an element of $\X'$, hence $f(p_j)\in \Y_j\cap \calU$ and $\proj(f(p_j))\in \Y_j\cap \X$.  Consequently there is $g_j\in D$ such that $p_j=\a_{g_j}(\proj(f(p_j)))$. Necessarily $g_j\in \isotropy$.

Using that $\isotropy$ is compact we may, passing to subsequences, assume that $\set{g_j}$ converges to some $g_0$. Denote $h_j=g_jg_0^{-1}$; these are elements of $\isotropy$. The sequence $\set{h_j}$ converges to $e$, the identity element of $\isotropy$.  Replace $f$ by $\a_{g_0}f$. Since $g_0\in \isotropy$, this replacement does not affect the already made assumption that $f$ is the identity on $\Y$ and we now have
\begin{equation*}
h_j\in \isotropy,\ \a_{h_j}(\proj (f(p_j)))=p_j\text{ for all }j, \text{ and }h_j\to e\text{ as }j\to \infty.
\end{equation*}

Write $\mathfrak e$ for the diffeomorphism $\exp_{p_0}\big|_B:B\to \X$, define
\begin{gather*}
\tilde \a_h:B'\to B',\ \tilde \a_h=\mathfrak e^{-1}\circ \a_h\circ \mathfrak e\ \text{ for }h\in \isotropy,\\
 \tilde f:B'\to B,\ \tilde f=\mathfrak e^{-1}\circ \proj \circ f\circ  \mathfrak e.
\end{gather*}
Letting $v_j=\mathfrak e^{-1}(p_j)$  (an element of $B'$) we have
\begin{equation*}
\tilde \a_{h_j}(\tilde f(v_j))=v_j\text{ for all }j\text{ and }v_j\to 0.
\end{equation*}
The sequence with elements $\hat v_j=|v_j|^{-1}v_j$ again has a convergent subsequence. Again passing to subsequences, assume that it converges, let $\hat v$ be its limit. We will show that 
\begin{equation*}
\frac{1}{|v_j|}(\tilde \a_{h_j}(\tilde f(v_j))-v_j)\to \tilde f(\hat v)-\hat v.
\end{equation*}
Since the terms on the left hand side vanish for all $j$, $\tilde f(\hat v)=\hat v$, which by elementary linear algebra means that \eqref{DetCondition} has a nonzero solution. 

Both $\tilde f$ and the functions $\tilde \a_h$, $h\in \isotropy$ vanish at $0$, so their first order Taylor expansion with remainder is
\begin{equation*}
\tilde f(v)=\tilde f^{(1)}(v)+R_{\tilde f}(v),\quad \tilde \a_h(v)=\tilde a_h^{(1)}(v)+R_{\tilde a_h}(v)
\end{equation*}
with quadratic estimates of the remainders, uniform in $h$ in the case of $\tilde \a_h$. Using the linearity of $\tilde f^{(1)}$ and $\tilde a_h^{(1)}$ we obtain
\begin{equation*}
\tilde a_h(\tilde f(v))=\tilde a_h^{(1)}(\tilde f^{(1)}(v))+R_h(v)
\end{equation*}
with 
\begin{equation*}
R_h(v)=\tilde a_h^{(1)}(R_{\tilde f}(v))+R_{\tilde a_h}(\tilde f(v))=\Oh(|v|^2).
\end{equation*}
Thus 
\begin{equation*}
0=\frac{1}{|v_j|}(\tilde \a_{h_j}(\tilde f(v_j))-v_j)=(\tilde \a_{h_j}^{(1)}(\tilde f^{(1)}(\hat v_j))-\hat v_j)+ \frac{1}{|v_j|}R_{h_j}(v_j)
\end{equation*}
which implies, passing to the limit as $j\to \infty$, that $\tilde f^{(1)}(\hat v)-\hat v=0$.
\end{proof}


\section{Limits and wave front sets}\label{Limits}

We now return to the setup of Section \ref{L2Traces} and define specific operators $\Phi^q_s$ satisfying the properties in that section leading to \eqref{LefshetzAsL2trace}. We use Kotake's approach \cite{Kotake} using the heat kernel as mollifier which also has the advantage of providing a way of computing the limit under the assumption that \eqref{WFCondition} holds.

Since $P_q$, defined in \eqref{FiniteCohomology}, is a non-negative elliptic operator that commutes with $\Lie_\T=\nabla_\T$, the heat operator $e^{-sP_q}$ exists for $s\geq 0$ and commutes with $\Lie_\T$. Let $\Phi^q_s=\Phi^q\circ e^{-sP_q}$. The operator $\Phi^q_s$ is smoothing when $s> 0$ because $e^{-sP_q}$ is. This and the fact that no element $(\pmb \xi\oplus \pmb \xi')\in \WF(K_{\Phi^q})$ has $\pmb \xi'=0$ give that $K_{\Phi^q_s}$ is smooth. Thus \eqref{CommuteAndSmoothing} holds. Using the representation of $e^{-sP_q}$ in terms of projections on the joint eigenspaces of $P_q$ and $-\im\Lie_T$ one gets that also \eqref{sGeometricEndo} and \eqref{PhiAt0} are satisfied; for the latter observe that $e^{-s P_q}$ is the identity on $\Ha^q$. Thus all conclusions in Section \ref{L2Traces} hold. 

We now determine convergence as $s\to 0^+$. This will follow from \cite[Theorem 2.5.11$'$, p.~128]{Hor71} and a property of the heat kernel by way of writing the Schwartz kernel of the composition as a succession of pull-back and push-forward operations. 

For an arbitrary manifold $\Z$ and closed conic subset $W\subset T^*\Z\minus 0$, the space 
\begin{equation*}
C^{-\infty}_W(\Z)=\set{u\in C^{-\infty}(\Z):\WF(u)\subset W}
\end{equation*}
with the topology given by all the seminorms of $C^{-\infty}(\Z)$ together with all seminorms $u\mapsto a(Pu)$ where $P$ is a pseudodifferential operator which is microlocally smoothing on $W$ and $a$ is a seminorm of $C^\infty(\Z)$, is a complete topological vector space containing $C^\infty(\Z)$ as a dense subspace. With this topology, if $\X$ is a smooth manifold and $\iota:\X\to \Z$ is smooth and such that $\iota^*W\subset T^*\X$ contains no zero covector, then the restriction operator
\begin{equation*}
\iota^*:C^\infty(\Z)\to C^\infty(\X)
\end{equation*}
is continuous in the respective subspace topology and therefore extends to a continuous operator 
\begin{equation*}
\iota^*:C^{-\infty}_W(\Z)\to C^{-\infty}_{\iota^*W}(\X).
\end{equation*}
Also, for arbitrary manifolds $\Z$, $\Y$, if $\pi:\Z\to \Y$ is a submersion with compact fibers, then $\pi_*$ is well defined on distributions (ignoring densities) and for an arbitrary closed conic subset $W\subset T^*\Z$
\begin{equation*}
\pi_*:C^\infty_W(\Z)\to C^\infty_{\pi_*W}(\Y),\quad \pi_*(W)=\set{\pmb \eta\in T^*\Y:\exists\pmb \zeta\in W,\ \pi^*\pmb \eta=\pmb \zeta}
\end{equation*}
is continuous. 

Let $\pi_\rmL$, $\pi_\rmR$ be the left and right projections, 
\begin{equation*}
\N\times \N\xleftarrow{\pi_\rmL}(\N \times \N)\times(\N\times \N) \xrightarrow{\pi_\rmR}
\N\times\N,
\end{equation*}
$\iota_{\Delta_{\mathrm M}}:\N\times\N\times \N\to (\N \times \N)\times(\N\times \N)$ the middle diagonal inclusion map
\begin{equation*}
\iota_{\Delta_{\mathrm M}}(p,p',p'')=(p,p',p',p''),
\end{equation*}
and 
\begin{equation*}
\pi:\N\times\N\times\N\to \N \times \N
\end{equation*}
the projection $\pi(p,p',p'')=(p,p'')$. Let $W_\rmL,W_\rmR\subset T^*(\N\times \N)\minus 0$ be closed conic subsets, let 
\begin{equation*}
W_{\rmL\rmR}=\pi_\rmL^*W_\rmL\cup\pi_\rmR^*W_{\rmR}\cup (\pi_\rmL^*W_\rmL\times\pi_\rmR^*W_{\rmR}),
\end{equation*}
and suppose 
\begin{equation*}
\iota_{\Delta_{\mathrm M}}^*W_{\rmL\rmR}\subset T^*(\N\times\N\times \N)
\end{equation*}
contains no zero covector. Finally let 
\begin{equation*}
W=\pi_*(\iota_{\Delta_{\mathrm M}}^*W_{\rmL\rmR}).
\end{equation*}

\begin{lemma}
Under the above conditions, the map
\begin{multline*}
C^{-\infty}_{W_\rmL}(\N\times\N)\times C^{-\infty}_{W_\rmR}(\N\times\N)\ni (K_\rmL,K_\rmR)\\\mapsto \pi_*(\iota_{\Delta_{\mathrm M}}^*(\iota_{\rmL}^*K_\rmL\otimes \iota_{\rmR}^*K_\rmR))\in C^{-\infty}_W(\N\times\N)
\end{multline*}
is continuous.
\end{lemma}

\begin{proposition}\label{ConvergenceViaHeat}
Let $W=\WF(K_{\Phi^q\circ \Av^q})$. If \eqref{WFCondition} holds, then 
\begin{equation*}
\iota_\Delta^*K_{\Phi^q\circ \Av^q}\in C^{-\infty}_{\iota_\Delta^*W}(\N; \iota_\Delta^* (E^q\boxtimes E^{q,*}))
\end{equation*}
is well defined. Furthermore
\begin{equation}\label{ThePullBack}
\iota_\Delta^* K_{\Phi^q\circ \Av^q} =\lim _{s\to 0^+}\iota_\Delta^*K_{\Phi^q_s\circ \Av^q}\text{ in }C^{-\infty}_{\iota_\Delta^*W}(\N;\iota_\Delta^*(E^q\boxtimes E^{q,*})).
\end{equation}
\end{proposition}

\begin{proof}
The first assertion is a direct application of H\"ormander's theorem cited above. The second, crucial for us, is also a consequence of that theorem via the just stated lemma, as follows. The fact that 
\begin{equation*}
\lim_{s\to 0^+}K_{e^{-sP_r}}= K_{\Id}\text{ in }C^{-\infty}_{\WF(K_\Id)}(\N\times\N;E^q\boxtimes E^{q,*})
\end{equation*}
implies 
\begin{equation*}
\lim_{s\to 0^+}K_{\Phi^q_s}= K_{\Phi^q}\text{ in }C^{-\infty}_{\WF(K_{\Phi^q})}(\N\times\N;E^q\boxtensor E^{q,*})
\end{equation*}
which in turn implies
\begin{equation*}
\lim_{s\to 0^+}K_{\Phi^q_s\circ \Av^q}= K_{\Phi^q\circ\Av^q}\text{ in }C^{-\infty}_{\WF(K_{\Phi^q\circ \Av^q})}(\N\times\N;E^q\boxtimes E^{q,*}).
\end{equation*}
This and the lemma give \eqref{ThePullBack}.
\end{proof}

Since 
\begin{equation*}
\sum_{q=0}^m (-1)^q\Tr(\Phi^q_s\circ\Av^q)
\end{equation*}
is independent of $s>0$ (see \eqref{LefshetzAsL2trace}) and each term converges as $s\to 0^+$ when \eqref{WFCondition} holds, we have
\begin{corollary}
If \eqref{WFCondition} holds, then 
\begin{equation*}
L_{f,\phi} = \sum_{q=0}^m (-1)^q \langle \iota_\Delta^* K_{\Phi^q\circ \Av^q},\calI_\tensor^q\rangle.
\end{equation*}
\end{corollary}


\section{Approximating sequences}

We return to the general setting of Sections~\ref{ConnectionsGroupsAverage}, \ref{WaveFront}, and \ref{discreteness}. Let $\Omega_\rmR\to\N\times \N$ be the pull-back of $\Omega\to\N$, the bundle of $1$-densities of $\N$, through the right projection $\pi_\rmR:\N\times\N \to \N$. The Schwartz kernel of  $f^*:C^\infty(\N)\to C^\infty(\N)$ is a generalized section of $\Omega_{\rmR}$ supported along $\Gamma$, the graph of $f$,
\begin{equation*}
K_{f^*}\in C^{-\infty}(\N\times \N;\Omega_{\rmR}).
\end{equation*}
If $\phi_\tensor$ is a smooth section of $F\boxtimes F^*$ extending the section along $\Gamma$ corresponding to the map $\phi$ (see \eqref{phiAsTensor}), then 
\begin{equation*}
\phi_\tensor K_{f^*}
\end{equation*}
is the Schwartz kernel of $\Phi:C^\infty(\N;F)\to C^\infty(\N;F)$, see \eqref{DefOfPhiq}. So get a hold on the latter we only need to get a good handle on $K_{f^*}$ which we will get through the use of a simple mollifier, rather than a heat kernel. If in the resulting approximation the map $f$ is replaced by the identity map, we get (of course) an approximation to the Schwartz kernel of the identity map.

We will define a sequence of smooth sections of $\Omega_{\rmR}$ converging to $K_{f^*}$ (acting on functions). From this we will also easily get a sequence $\set{K_{\Phi\circ \Av,k}}_{k=1}^\infty$ of smooth sections of $F\boxtimes F^*\otimes \Omega_{\rmR}$ converging to $K_{\Phi\circ \Av}$. The important point is that under the hypothesis \eqref{WFCondition} also $\set{\iota_\Delta^*K_{\Phi\circ \Av,k}}_{k=1}^\infty$ will converge.

\medskip

Let $\gg$ be a $\T$-invariant Riemannian metric on $\N$. As in Section~\ref{discreteness} let $r>0$ be such that metric balls of radius $r$ are strictly convex. For each $p\in \N$ let $B_p\subset T_{(p,f(p))}(\N\times \N)$  be the ball of radius $r$ in $0_p\oplus T_{f(p)}\N$. The collection of these balls is a subbundle $B\to \Gamma$ of the part $T_\Gamma(\N\times \N)$ of $T(\N\times\N)$ over the graph of $f$.  Define
\begin{equation*}
\mathfrak e:B\to \N\times\N,\quad \mathfrak e(\pmb v)=\exp(\pmb v), \ \pmb v\in B.
\end{equation*}
Thus $\mathfrak e$ is a diffeomorphism onto some neighborhood $\calW$ of $\Gamma$. Let 
\begin{equation}\label{DefOfGamma}
\gamma:\calW\to B
\end{equation}
be the inverse. 

Denote by $\m_{\rmR}$ the section of $\Omega_{\rmR}$ determined by $\m$: $\m_\rmR=\pi_\rmR^*\m$. Let $\chi_0\in C_c^\infty(\R)$ be non-negative, supported on $(-r,r)$, positive near $0$ and let $\chi:B\to \R$ be given by $\chi(\pmb v)=\chi_0(|\pmb v|)$. 
Define $K_{f^*,k}$ as a section of $\Omega_{\rmR}\to \N\times\N$ by setting, for $k=1,2\dots$, 
\begin{equation*}
K_{f^*,k}(p,p')=k^n\chi(k\,\gamma(p,p'))c(\pi(\gamma(p,p')))\m_{\rmR}(p')
\end{equation*}
if $(p,p')\in \calW$ and $0$ if not, where $\pi:T_\Gamma(\N\times \N)\to \Gamma$ is the canonical projection. The function $c:\Gamma\to \R$, to be specified once we show the convergence of $K_{f^*,k}$, is smooth and positive.

\begin{lemma}\label{KfPullBack}
The sequence $\set{K_{f^*,k}}_{k=1}^\infty$ converges as a distribution.
\end{lemma}

\begin{proof}
Define $\Omega_\rmL$ in the way $\Omega_\rmR$ was defined, now using the left projection. If $w\in C^\infty(\N\times\N;\Omega_{\rmL})$ is arbitrary then $w K_{f^*,k}$ is a smooth density on $\N\times \N$. We show that
\begin{equation}\label{TestingId}
\int_{\N\times \N}w K_{f^*,k}
\end{equation}
converges. Let $p_0\in \N$ be arbitrary, $\calU'$ a neighborhood of $f(p_0)$ and $\calU$ one of $p_0$ chosen so that $f(\calU)\subset \calU'$. If $w\in C^\infty(\N\times\N;\Omega_{\rmL})$ is supported in $\calU\times \calU'$ and $w(p,p')=w_0(p,p')\m(p)$, then \eqref{TestingId} reads
\begin{equation}\label{AsAnIntegral}
\int_{\N\times \N}w_0(p,p')k^n\chi(k\,\gamma(p,p'))c(\pi(\gamma(p,p')))\,d\m(p)\,d\m(p').
\end{equation}
We will evaluate the limit as $k\to \infty$ by first introducing the change of variables $\pmb v=\gamma(p,p')$ and using local coordinates.

We choose the neighborhoods $\calU$, $\calU'$ so small that they are domains of local charts,  $x^1,\dots,x^n$ in the first, ${x'}^1,\dots,{x'}^n$ in the second. We get coordinates $(y,\eta)$ on the part of $0\oplus T\N$ over $\Gamma\cap(\calU\times \calU')$ by letting $y^j={x'}^j\circ \pi(\pmb v)$ and $(\eta^1,\dots,\eta^n)$ so that
\begin{equation*}
\pmb v=\sum \eta^i\frac{\partial}{\partial {x'}^i}\Big|_{y,f(y)}.
\end{equation*}
In these coordinates
\begin{equation*}
\mathfrak e(y,\eta)=\big (X(y,\eta),X'(y,\eta)\big ) = (y,f(y)+\eta+\Oh(|\eta|^2))
\end{equation*}
and so
\begin{equation*}
\gamma(x,x')=(Y(x),\Eta(x,x'))=\big(x,x'-f(x)+\Oh(|x'-f(x)|^2)\big)
\end{equation*}

Writing the integrand in \eqref{AsAnIntegral} in the coordinates $(x,x')$ using $\m_0(x)|dx|$ and $\m_0'(x')|dx'|$ for the densities and introducing the change of variables $x=X(y,\eta)$, $x'=X'(y,\eta)$, gives
\begin{equation*}
w_0(X(y,\eta),X'(y,\eta))k^n \chi(k\pmb v(\eta))c(\pi(\pmb v(\eta)))\m_0(X(y,\eta))\m_0'(X'(y,\eta))\,dy\, d\eta
\end{equation*}
where $\pmb v=\sum \eta^i\partial_{{x'}^i}$. The further change of variables $\eta=k^{-1}\tilde \eta$ followed by integration yields
\begin{equation*}
\lim_{k\to \infty}\int_{\N\times\N} wK_{f^*,k}=\\
\int_{B}w_0(y,f(y))\chi(\pmb v(\eta))c(y,f(y))\m_0(y)\m_0(f(y))\,dy\, d\eta.
\end{equation*}
We reorganize the integrand on the right observing that $w_0(x,f(y))\m_0(y)\,dy=\iota_{\Gamma}^*w$ to get
\begin{equation}\label{TheLimitTestingId}
\int_\N c(p,f(p))\Big[\int_{B_{(p,f(p))}}\chi(\pmb v(\eta))\m_0(f(p))\,d\eta\Big]\,\iota_{\Gamma}^*w
\end{equation}
Thus \eqref{TestingId} converges as $k\to \infty$.
\end{proof}

The density $\m_0\,d\eta$ in \eqref{TheLimitTestingId} is independent of the coordinates used. Keeping the notation of the proof of the lemma,  is $\tilde x'{}^1,\dots \tilde x'{}^n$  is another coordinate system near $f(p_0)$, then
\begin{equation*}
\pmb v=\sum\eta^i \frac{\partial }{\partial {x'}^i} = \sum_j \tilde \eta^j \frac{\partial }{\partial {\tilde x'}{}^j}
\end{equation*}
with
\begin{equation*}
\tilde \eta_j=\sum_i\eta_i \frac{\partial {x'}^i}{\partial \tilde x'{}^j},
\end{equation*}
therefore 
\begin{equation*}
\m_0(f(p))\,d\eta=\m_0(f(p))\big|\det\frac{\partial \tilde x'}{\partial x'}\big|\,d\tilde \eta
\end{equation*}
But $\m=\tilde \m_0\, d\tilde x'$ in the new coordinates so $\m_0(f(p))\,d\eta=\tilde \m_0(f(p))\,d\tilde \eta$. Consequently the expression in brackets in \eqref{TheLimitTestingId} is independent of coordinates. By choice it is positive everywhere so we can define $c:\Gamma\to \R$ in such a way that 
\begin{equation*}
\int_{\N\times \N} w K_{f^*,k}\to \int_\N \iota_{\Gamma}^*w\quad\text{ as }k\to\infty. 
\end{equation*}
For later use we note that the expression in brackets in \eqref{TheLimitTestingId} is constant on orbits of $G$, so
\begin{equation}\label{cIsConstantOnY}
c(\a_g(p),f(a_g(p)))=c(p,f(p)).
\end{equation}

\begin{proposition}\label{ApproximatePhi}
The sequence 
\begin{equation*}
\set{K_{\Phi,k}}_{k=1}^\infty,\quad  K_{\Phi,k}=K_{f^*,k}\phi_\tensor.
\end{equation*}
converges to $K_{\Phi}$. 
\end{proposition}

\begin{proof}
Since $\phi_\tensor$ is a smooth section of $F\boxtimes F^*$, so is $K_{f^*,k}\phi_\tensor$. If $w_0$ is a smooth section of $F^*\boxtimes F$, then $(p,p')\mapsto \langle \phi_\tensor(p,p'), w_0(p,p')\rangle$ is a smooth function and 
\begin{equation*}
\langle \phi_\tensor K_{f^*,k},w_0\m_\rmL\rangle = \big\langle K_{f^*,k},\langle \phi_\tensor, w\rangle \m_\rmL\big\rangle \to  \int_\N \big\langle \phi_\tensor(p,f(p)), w_0(p,f(p))\big\rangle\,d\m(p)
\end{equation*}
as $k\to \infty$. In particular, if $w_0=v_0\otimes u$ with $v_0\in C^\infty(\N;F^*)$ and $u\in C^\infty(\N;F)$, then the definition of $\phi_\tensor$ (see \eqref{phiAsTensor}) gives
\begin{equation*}
 \langle \phi_\tensor(p,f(p)), w_0(p,f(p))\rangle =  \langle v_0(p),\phi(p,u(f(p)))\rangle= \langle v_0(p),\Phi(u)(p)\rangle,
\end{equation*}
therefore $\langle \phi_\tensor K_{f^*,k},w\rangle\to \langle v,\Phi(u)\rangle=\langle K_\Phi,v\boxtimes u\rangle$. Thus $K_{f^*,k}\phi_\tensor$ converges to $K_{\Phi}$ as $k\to \infty$.
\end{proof}

Let $K_{\Phi,k,0}=K_{f^*,k,0}\phi_\tensor$, define $K_{\Phi\circ \Av,k,0}\in C^\infty(\N\times \N;F\boxtimes F^*)$ by
\begin{equation*}
K_{\Phi\circ \Av,k,0}(p,p')= \int_{\hat G} (\Id\boxtimes \A_{\hat g}^*)K_{\Phi,k,0}\big(p,\a_{\wp(\hat g)}(p')\big)\,d\hat \mu(\hat g).
\end{equation*}
and $K_{\Phi\circ \Av,k}=K_{\Phi\circ \Av,k,0}\m_R$.

\begin{proposition} 
We have
\begin{equation*}
\op(K_{\Phi\circ\Av,k})=\op(\phi_\tensor K_{f^*,k})\circ \Av.
\end{equation*}
The sequence $\set{K_{f^*\circ\Av,k}}_{k=1}^\infty$ converges to $K_{\Phi\circ \Av}$.
\end{proposition}

\begin{proof}
Let $w_0\in C^\infty(\N\times \N;F^*\boxtimes F)$, $w=w_0\,d\m_\rmL$. Then 
\begin{align*}
\langle &K_{\Phi\circ\Av,k},w\rangle \\
&= \int_{\N\times \N}\int_{\hat G}\big\langle (\Id\boxtimes \A_{\hat g}^*)K_{\Phi,k,0}\big(p,\a_{\wp(\hat g)}(p')\big),w_0(p,p')\big\rangle\,d\hat \mu(\hat g)\,d\m(p)\,d\m_(p')\\
&= \int_{\N\times \N}\int_{\hat G}\big\langle K_{\Phi,k,0}\big(p,\a_{\wp(\hat g)}(p')\big),(\Id\boxtimes \A_{\hat g})w_0(p,p')\big\rangle\,d\hat \mu(\hat g)\,d\m(p)\,d\m_(p')\\
&= \int_{\N\times \N}\int_{\hat G}\big\langle K_{\Phi,k,0}\big(p,p'\big),(\Id\boxtimes \A_{\hat g})w_0(p,\a_{\wp(\hat g^{-1})}p')\big\rangle\,d\hat \mu(\hat g)\,d\m(p)\,d\m_(p')\\
&= \int_{\N\times \N}\big\langle K_{\Phi,k,0}\big(p,p'\big),\int_{\hat G}(\Id\boxtimes \A_{\hat g})w_0(p,\a_{\wp(\hat g^{-1})}p')\,d\hat \mu(\hat g)\big\rangle\,d\m(p)\,d\m_(p').
\end{align*}
If $w_0=v_0\otimes u$, $v_0\in C^\infty(\N;F^*)$ and $u\in C^\infty(\N;F)$, then 
\begin{equation*}
\int_{\hat G}(\Id\boxtimes \A_{\hat g})w_0(p,\a_{\wp(\hat g^{-1})}p')\,d\hat \mu(\hat g) = v_0(p)\otimes \Av(u)(p')
\end{equation*}
hence
\begin{align*}
\langle K_{\Phi\circ\Av,k},w\rangle &= \int_{\N\times \N}\big\langle K_{\Phi,k,0}\big(p,p'\big),v_0(p)\otimes \Av(u)(p')\rangle \,d\m(p)\,d\m(p')\\
&=\langle v, \op(K_{\Phi,k})(\Av(u))\rangle
\end{align*}
with $v=v_0\,\m_{\rmL}$. This proves the first assertion. The second is immediate from this and the previous proposition.
\end{proof}

Note for later use that if $(p,p')\in \supp K_{\Phi\circ\Av,k}$, then there is $\hat g\in \hat G$ such that $K_{f^*,k}(p,\a_{\wp(\hat g)}(p'))\ne0$, i.e. $d(p,\a_{\wp(\hat g)}(p'))<r/k$. 

\medskip
Let $W=\WF(K_{f^*})$, see \eqref{WFofPhiAv}. From its definition we see that $\WF(K_{\Phi,k})=\WF(K_{f^*})$.
\begin{proposition}
The sequence $\set{K_{f^*,k}}_{k=1}^\infty$ converges to $K_{f^*}$ in $C^{-\infty}_{N^*\Gamma}(\N\times \N;\Omega_R)$. Consequently 
\begin{equation*}
K_{\Phi\circ \Av,k}\to K_{\Phi\circ \Av}\text{ in }C^{-\infty}_{W}(\N\times \N;(F\boxtimes F^*)\otimes \Omega_\rmR).
\end{equation*}
as $k\to \infty$.
\end{proposition}

\begin{proof}
The second statement will follow from the first by the same arguments as in Proposition \ref{ConvergenceViaHeat}. 

Working in coordinates and with notation in the proof of Lemma \ref{KfPullBack}, let $w$ be smooth supported in $\calU\times \calU'$. The Fourier transform $(K_{f^*,k} w)\hat{\ }(\upsilon,\theta)$ in the coordinates $(y,\eta)$ is an expression of the form
\begin{equation*}
\int \phi(y,\eta)k^n\chi(k\eta) e^{-\im(y\cdot \upsilon+\eta\cdot \theta)}\,dy\,d\eta
=\int \phi(y,\eta/k)\chi(\eta) e^{-\im(y\cdot \upsilon+\eta\cdot \theta/k)}\,dy\,d\eta
\end{equation*}
where $\phi$ incorporates unessential details from $w$ and $c$ except that it is smooth and compactly supported. We thus have the bound
\begin{equation*}
C_N(1+|\upsilon|+|\theta|/k)^{-N}
\end{equation*}
for any $N$ and some $C$, uniformly in $k$. In the region $|\theta|<\kappa|\upsilon|$, $\kappa>0$, this is equivalent to the bound
\begin{equation*}
C_{N,\kappa}(1+|\upsilon|+|\theta|)^{-N}.
\end{equation*}
In the coordinates we are using, the conormal bundle of $\Gamma$ is  $\set{\upsilon=0}$, the conormal bundle of the zero section of $N\Gamma$. Thus we have rapid decrease and uniform convergence in the complement of any closed cone disjoint from $N^*\set{\eta=0}$, which implies that the convergence takes place in $C^{-\infty}_{N^*\Gamma}(\N\times\N)$.
\end{proof}

\begin{corollary} 
If \eqref{WFCondition} holds, then the sequence $\set{\iota_\Delta^* K_{\Phi\circ \Av,k,0}}_{k=1}^\infty$ converges in $C^{-\infty}(\N;F\otimes F^*)$ as $k\to \infty$.
\end{corollary}

We find the limit in the next section.



\section{Limit}\label{sLimit}

We continue with the objects defined in the previous section, in particular smooth approximations $K_{\Phi\circ \Av,k}$, $k=1,2,\dots$ to $K_{\Phi\circ \Av}$. Assuming the transversality condition \eqref{WFCondition}, we will find a formula for $\iota_\Delta^*K_{\Phi\circ\Av}$.

Define
\begin{equation*}
\mathbf{Fix}=\bigcup_{\Y\in \Fix(\pmb f)} \Y.
\end{equation*}
As a union of finitely many closed sets, $\mathbf{Fix}$ is closed. The following lemma allows us to restrict the analysis to arbitrarily small neighborhood of the fixed orbits (as expected). 

\begin{lemma}
Let $A\subset \N$ be closed and disjoint from $\mathbf{Fix}$. There is $k_0$ such that $k>k_0$ implies $\supp(\iota_\Delta^*K_{\Phi\circ \Av,k})\cap A=\emptyset$.
Thus, if $k$ is large enough and $w\in C^\infty(\N;F^*\otimes F)$ has $\supp w\subset A$, then $\langle \iota_\Delta^*K_{\Phi\circ\Av,k},w\rangle=0$.
\end{lemma}

\begin{proof}
We use $d$ to denote distance both on $\N$ and on $\N\times \N$, the latter with the product metric. Since $A$ is compact and $A\cap \mathbf{Fix}=\emptyset$, the function $A\times G\ni (p,g)\mapsto d(p,a_g(f(p)))\in \R$ is bounded below by a positive number $c_0$. 

Let $k>r/c_0$. If $(p,p')\in \supp K_{\Phi\circ \Av,k}$, then $|\gamma(p,\a_{\wp(\hat g')}(p'))|<r/k$ for some $g'\in G$. Therefore if $(\a_{\wp(\hat g)}(p),p) \in \supp K_{\Phi\circ \Av,k}$, then $|\gamma\big(\a_{\wp(\hat g)}(p),\a_{{\wp(\hat g')}}(p)\big)|<r/k$ for some $\hat g'$. Thus $d\big(f(\a_{\wp(\hat g)}(p)),\a_{\wp(\hat g')}(p)\big)<r/k$. But since $G$ acts by isometries, this gives $d(p,\a_{g''}(f(p)))<r/k$ for some $g''$. But  $r/k<c_0$, so $p\notin A$. The second assertion is evident given the first.
\end{proof}

By the lemma, the support of $\lim_{k\to\infty} \iota_\Delta^* K_{\Phi\circ \Av,k}$ is contained in \begin{equation*}
\Delta\cap \bigcup_{\Y\in \mathbf{Fix}(f)} \Y\times \Y. 
\end{equation*}
Pick $\Y\in \Fix(\pmb f)$. We now aim at finding 
\begin{equation}\label{FindTheLimit}
\lim_{k\to\infty}\langle \iota^*_\Delta K_{\Phi\circ \Av,k},w\rangle
\end{equation}
for an arbitrary $w =w_0\,\m\in C^\infty(\N;F\otimes F^*)$ supported in a neighborhood of $\Y$ disjoint from $\Y' \in \Fix(\pmb f)$ if $\Y'\ne \Y$. 

In view of the definition of $K_{\Phi,k}=K_{\Phi,k,0}\m$
\begin{multline*}
\langle \iota^*_\Delta K_{\Phi\circ \Av,k},w\rangle\\
=\int_\N \int_{\hat G} K_{f^*,k,0}(p,\a_{\wp(\hat g)}(p))\big\langle (I\otimes \A_{\hat g}^*)(\phi_\tensor(p,\a_{\wp(\hat g)}(p))), w_0(p)\big\rangle\,d\hat\mu(\hat g) \,d\m(p).
\end{multline*}
Defining $\mathrm w:\N \times \hat G\to \C$ by
\begin{equation*}
\mathrm w(p,\hat g)=\big\langle (I\otimes \A_{\hat g}^*)(\phi_\tensor(p,\a_{\wp(\hat g)}(p))), w_0(p)\big\rangle,
\end{equation*}
a smooth function, we rewrite the above as
\begin{equation}\label{FinallyPullBack}
\langle \iota_\Delta^* K_{\Phi^*\circ \Av,k},w_0\,\m\rangle
=\int_\N \int_{\hat G} K_{f^*,k,0}\big(p,\a_{\wp(\hat g)}(p)\big) \mathrm w(p,\hat g)\,d\hat \mu(\hat g)\,d\m(p).
\end{equation}
Pick $g_0\in G$ such that  $f_0=\a_{g_0}\circ f$ fixes $\Y$ pointwise, let $\hat g_0$ be such that $g_0=\wp(\hat g_0)$, write the right hand side in \eqref{FinallyPullBack} using $f_0$: since 
\begin{equation*}
|\gamma(p,p')|=d(f(p),p')=|\gamma(f_0(p),\a_{g_0}(p'))|
\end{equation*}
if $(p,p')\in \calW$ (see \eqref{DefOfGamma}),
\begin{equation*}
K_{f^*,k,0}(p,p')=K_{f_0^*,k,0}(p,\a_{g_0}p')
\end{equation*}
with $K_{f_0^*,k,0}$ defined as $K_{f^*,k,0}$ was but with $f_0$ in place of $f$. Recall from \eqref{cIsConstantOnY} that the normalization constant $c(p,f(p))$ used in the definition of $K_{f^*,k,0}$ is constant on orbits of $G$. With this and a change of variable of integration, \eqref{FinallyPullBack} becomes
\begin{equation}\label{Kforf_0}
\int_\N \int_{\hat G} K_{f_0^*,k,0}\big(p,\a_{\wp(\hat g)}(p)\big) \mathrm w(p,\hat g\hat g_0^{-1})\,d\hat \mu(\hat g)\,d\m(p)
\end{equation}
We will find its limit as $k\to \infty$ with the aid of a partition of unity subordinate to open sets of the kind described in the next paragraph taking advantage of the setup preceding the proof of Theorem \ref{DiscreteFixedPoints}.

\smallskip
Pick $p_0\in \Y$ and let $r$, $B$, $d=\dim\Y$, and $\X$ be as in Section \ref{discreteness}. Assume $r$ so small that $\Y\cap \X=\set{p_0}$ and $\X$ does not intersect any of the other fixed orbits of $f$. Let $\isotropy\subset G$ be the isotropy group of $p_0$, let $\preisotropy =\wp^{-1}(\isotropy)\subset \hat G$, a compact subgroup. There is a compact connected subgroup $\hat G_0\subset \hat G$ transversal to $\preisotropy$ at the identity. Since the codimension of $\preisotropy$ in $\hat G_0$ is equal to that of $\isotropy$ in $\hat G$, $G_0=\wp(\hat G_0)$ is locally isomorphic (through $\wp$) to $\hat G_0$ ($\hat G_0\to G_0$ is a finitely sheeted cover of $G_0$). Let $\hat D$ be a ball in $\hat G_0$ about the identity element with respect to some invariant metric and let $D=\wp(\hat D)$ (see Example \ref{ExampleOfAction}). Assume $\hat D$ closed under inversion and so small that the translates of $\hat D \cdot \hat D$ by the elements of $\hat G_0\cap \preisotropy$ are pairwise disjoint ($D$ is evenly covered). Further, assume $\hat D$ small enough that, as in Section \ref{discreteness}, the map $\X\times D \ni (p,g)\mapsto \a_g p\in \N$ is a diffeomorphism. Let $\calU=\set{\a_g(p):p\in \X,\ g\in D}$, $\proj:\calU\to \X$ the projection as in Section \ref{discreteness}.

Further, let $\X''=\exp_{p_0}(B'')$ where $B''\subset N_{p_0}\Y$ is the ball of radius $r''<r$, let $\hat D''\subset \hat D$ be another ball about the identity, closed under inversion and such that $\hat D''\cdot \hat D''\subset \hat D$. Finally, choose $r'<r$ and $\hat D'\subset \hat D''$ with the requirement that the image of $\calU'=\set{\a_g(p):p\in \X',\ g\in \wp(D')}$ by $f_0$ is contained in $\calU''=\set{\a_g(p):p\in \X'',\ g\in \wp(\hat D'')}$. Observe that every neighborhood of $p_0$ in $\N$ contains a neighborhood of $p_0$ like $\calU'$.

\begin{lemma}
Let $\calK\Subset \calU'$ be compact. There is $k_0$ such that if $k>k_0$, $p\in \calK$ and $K_{f_0^*,k,0}(p,\a_g(p))\ne 0 $ implies $g\in D\cdot \isotropy$.
\end{lemma}

\begin{proof}
Since $\calK \Subset \calU'$, $f_0(\calK)\Subset \calU''$, so there is $k_0$ such that $d(f_0(\calK),p')<r/k$ implies $p'\in \calU''$ whenever $k> k_0$. Assuming this of $k$, we have $p\in \calK$ and $d(f_0(p),\a_g(p))<r/k$ imply $\a_g(p)\in \calU''$, and thus $\a_g(p)=\a_{g''}(p'')$ for some $(p'',g'')\in \X''\times D''$. Since $p=\a_{g'}(\proj(p))$ for some $g'\in D'$, $\a_{gg'}(\proj(p))=\a_{g''}(p'')$, hence $p''=\a_{gg'{g''}^{-1}}(\proj(p))$. Since $p''$ and $\proj (p)$ both lie in $\X$, $gg'{g''}^{-1}\in \isotropy$. But $g'{g''}^{-1}\in D''\cdot D''\subset D$, hence $g\in D\cdot \isotropy$.
\end{proof}

It follows that if $k$ is large enough and $p\in \calK$, then $K_{f_0^*,k,0}\big(p,\a_{\wp(\hat g)}(p)\big)\ne 0$ implies $\hat g\in \hat D\cdot \preisotropy$. 

\begin{lemma}
Let $\hat g\in \hat D\cdot \preisotropy$. The factorization $\hat g=\hat g'\hat h$ with $\hat g'\in \hat D$ and $\hat h\in \preisotropy$ is unique.
\end{lemma}

\begin{proof}
Suppose $\hat g\hat h=\hat g'\hat h'$ with $\hat g,\hat g'\in \Hat D$ and $\hat h,\hat h'\in \preisotropy$. Then $\hat h^{-1}\hat h'=\hat g{\hat g'}{}^{-1}\in (\hat D\cdot \hat D)\cap \preisotropy$. Since the translates of $\hat D\cdot \hat D$ by the elements of $\preisotropy_{p_0}\cap \hat G_0$ are pairwise disjoint, $\hat g{\hat g'}{}^{-1}$ is the identity element of $\hat G_0$.
\end{proof}

We now return to the evaluation of the integral in \eqref{Kforf_0}, replacing $\mathrm w$ by $\vartheta\mathrm w$ where $\vartheta\in C_c^\infty(\calU')$ plays the role of an element of a partition of unity. We let $\hat \mu_\preisotropy$ be the normalized Haar measure on $\preisotropy$ and $\hat \mu_{\hat G_0}$ the Haar measure on $G_0$ such that $\hat\mu_{\hat G}=\hat \mu_\preisotropy\otimes \hat \mu_{\hat G_0}$. The following proposition is immediate in view of the last two lemmas. We use that $\a_{\wp(\hat h)}$ is an isometry.

\begin{proposition}
Let $\vartheta\in C_c^\infty(\calU')$. There is $k_0$ such that 
\begin{multline}\label{findLimit}
\langle \iota_\Delta^* K_{\Phi^*\circ \Av,k},\vartheta w\rangle \\=\int_{\preisotropy}\int_{\calU'}\int_{\hat D} K_{f_0^*,k,0}\big(p,\a_{\wp(\hat g\hat h)}(p)\big) \vartheta(p)\mathrm w(p,\hat g\hat h\hat g_0^{-1})\,d\hat \mu_{\hat G_0}(\hat g)\,d\m(p)\,d\mu_\preisotropy(\hat h)
\end{multline}
if $k>k_0$.
\end{proposition}

We will rewrite \eqref{findLimit} using coordinates. Let $d$ be the dimension of $\Y$. Pick an orthonormal basis $\pmb v_1,\dots \pmb v_{n-d}$ of $N_{p_0}\Y$ and let $x^1,\dots x^{n-d}:\calU\to \R$ be such that 
\begin{equation*}
\proj(p)=\exp_{p_0}\big(\sum_j x^j(p)\pmb v_j\big), \quad p\in \X.
\end{equation*}
In particular, the $x^j$ are constant on the fibers of $\proj$. Next, identifying the ball $\hat D\subset \hat G_0$ linearly with a ball centered at $0$ in $\R^d$ (really the Lie algebra of $\hat G_0$ via the exponential map), let $\tilde y^1,\dots, \tilde y^d$ be the resulting coordinates on $\hat D$, transferred to functions $y^1,\dots,y^d$ on $\calU$  via the diffeomorphism $\calU\approx\X\times \hat D$, i.e. setting $y^j(p)=\tilde y^j(\hat g)$ if $p=\a_{\wp(\hat g)}(\proj(p))$. In these coordinates $d\mu_{\hat G_0}=\kappa|d\tilde y|$ on $\hat D$ with some constant $\kappa$. 

Using these coordinates (both on $\calU'$ and $\calU$), if $(x,y)$ corresponds to a point in $\calU'$, then $f_0(x,y)=(X_{f_0}(x,y),Y_{f_0}(x,y))$ with $X_{f_0}(x,y)$ vanishing at $x=0$ and $Y_{f_0}(0,y)=y$ (thus preserving $\Y$). Further, $X_{f_0}$ is independent of $y$, because $f_0$ commutes with $\a_g$, $g\in G_0$. For the same reason,  the function $Y_{f_0}$ is of the form $Y_{f_0}(x,y)=Y_{f_0}(x)+y$, where we are using additive notation for the operation on $G_0$ in the neighborhood $D'\subset D$ of the identity in $G_0$. Since $\Y$ is pointwise fixed by $f_0$, we further have $Y_{f_0}(0)=0$. Thus, using additive notation for the operation on $G_0$ and identifying $y$ with an element of $D$,
\begin{equation*}
f_0(x,y)=(X_{f_0}(x),Y_{f_0}(x)+y)
\end{equation*}
with $X_{f_0}$ and $Y_{f_0}$ as functions of $x$ only, both vanishing when $x=0$.

Next, if $\hat g\in \hat D$, then $\a_{\wp(\hat g)}(x,y)=(x,y+\wp(\hat g))$, and if $\hat h\in \preisotropy_{p_0}$, then $\a_{\wp(\hat h)}(x,y)=(X_{\a_{\wp(\hat h)}}(x),y)$ because $\isotropy$ preserves each slice $\a_g\X\subset \calU$, $g\in D$. 

Observe that since $G_0$ acts as isometries of the metric $\gg$, the tangent spaces of the orbits of $G_0$ and those of the slices $\a_g(\X)$ are mutually orthogonal. Therefore 
\begin{equation*}
\gg=\sum_{i,j=1}^{n-d} \gg_{ij}(x)\,dx^i\otimes dx^j+\sum_{a,b=1}^d \gg_{ab}(x)\, dy^a\otimes dy^b.
\end{equation*}
The coefficients are independent of $y$ since $\a_g^*\gg=\gg$ for all $g$.

\smallskip
 
Changing  $p$ to $\a_{\wp(\hat h^{-1})}(p)$ in \eqref{findLimit} preserves $\calU'$. Doing this and replacing the definition of $K_{f_0^*,k,0}$ the resulting integrand is 
\begin{equation}\label{findLimitIntegrand}
c(p)k^n\chi_0\big(kd(\a_{\wp(\hat h^{-1})}\circ f_0(p),\a_{\wp(\hat g)}(p))\big) \vartheta(\a_{\wp(\hat h^{-1})}(p))\mathrm w(\a_{\wp(\hat h^{-1})}(p),\hat h \hat g\hat g_0^{-1})
\end{equation}
with $p\in \calU'$, $\hat g\in \hat D$, and $\hat h\in \preisotropy$. In coordinates, with $(x,y)$ corresponding to $p$, 
\begin{equation*}
\a_{\wp(\hat h^{-1})}\circ f_0(x,y)=(X_{\a_{\wp(\hat h^{-1})}}\circ X_{f_0}(x),Y_{f_0}(x)+y)=(X_{\a_{\wp(\hat h^{-1})}\circ f_0}(x),Y_{f_0}(x)+y)
\end{equation*}
and
\begin{equation*}
\a_{\wp(\hat g)}(p)=\a_{\wp(\hat g)}(x,y)=(x,y+\tilde y),\quad \tilde y=\wp(\hat g),
\end{equation*}
(viewing $y, \tilde y \in D\approx\wp(\hat D)$) so if $p\in \calU'$ and $d(\a_{\wp(\hat h^{-1})}\circ f_0(p),\a_{\wp(\hat g)}(p))<r/k$ with $k$ large, then
\begin{multline*}
d(\a_{\wp(\hat h^{-1})}\circ f_0(p),\a_{\wp(\hat g)}(p))^2 =\\ \sum_{i,j}\gg_{ij}(x)(X_{\wp(\a_{\hat h^{-1}})\circ f_0}-x^i)(X_{\wp(\a_{\hat h^{-1}})\circ f_0}^j(x)-x^j)\\
+\sum_{a,b}\gg_{ab}(x)(Y_{f_0}^a(x)-\tilde y^a)(Y_{f_0}^b(x)-\tilde y^b)+\epsilon
\end{multline*}
where $\epsilon=\Oh(1/k^3)$ incorporates the cross terms.  The transversality assumption (Proposition \ref{WFvsDet}) gives that 
\begin{equation*}
x\mapsto z(x)=X_{\a_{\wp(\hat h^{-1})}\circ f_0}(x)-x\text{ is invertible near }x=0.
\end{equation*}
Introducing the change of variable $z=z(x)$ followed by $z =\tilde  z/k$ and $\tilde y=y'/k$ in the inner iterated pair of integrals in \eqref{findLimit} with fixed $\hat h\in \preisotropy$ and relabeling get 
\begin{multline}\label{AlmostDone}
\int_D\int_{\R^{n-d}}\int_{\R^d} c(x(\frac{z}{k}))\chi_0\big(\Psi_k(z,y'-kY_{f_0}(x(\frac{z}{k})))\big)\vartheta\big(X_{\a_{\wp(\hat h^{-1})}}(x(\frac{z}{k})),y\big) \\
\x \frac{\mathrm w\big((X_{\a_{\wp(\hat h^{-1})}}(x(\frac{z}{k})),y),\hat h\hat g_0^{-1}+\frac{y'}{k}\big) \kappa\, \m_0(x(\frac{z}{k}))}{|\det\big[JX_{\a_{\wp(\hat h^{-1})}\circ f_0}(x(\frac{z}{k}))-\Id\big]|}\,dy'\,dz\,dy.
\end{multline}
The expression $\hat h\hat g_0^{-1}+y'/k$ means $\hat h\hat g_0^{-1}\Exp(y'/k)$, $J X_{\a_{h^{-1}}\circ f_0}$ is the Jacobian matrix of the transformation $x\mapsto X_{\a_{h^{-1}}\circ f_0}(x)$ with respect to the coordinates $x^1,\dots,x^{n-d}$ both in the domain and the codomain, 
\begin{equation*}
\Psi_k(z,\tilde y)^2=\sum \gg_{ij}\big(x(\frac{z}{k})\big)z^i z^j+\sum \gg_{ab}\big(x(\frac{z}{k})\big)\tilde y^a\tilde y^b+\epsilon
\end{equation*}
now with $\epsilon=O(1/k)$, and  $d\m=\m_0(x)\,dx\,dy$ incorporating the invariance of $\m$; because of \eqref{cIsConstantOnY}, the normalization factor $c$ is independent of $y$. 

The change of variable $y'=\tilde y+kY_{f_0}(x(z/k))$ reduces the argument of $\chi_0$ to $\Psi_k(z,\tilde y)$. The integral is over a compact set. The limit of $\Psi_k(z,\tilde y)$ as $k\to\infty$ is the norm of 
\begin{equation*}
\sum_j z^j \frac{\partial}{\partial x^j}+\sum_a \tilde y^a\frac{\partial}{\partial y^a}\in T_{p_0}\N
\end{equation*}
with respect $\gg$, therefore
\begin{equation*}
\lim_{k\to\infty} c\big(x(\frac{z}{k})\big)\int \chi_0(\Psi_k(z,\tilde y))\,\m\big(x(\frac{z}{k})\big)\,d\tilde y=1.
\end{equation*}
It follows that the limit of \eqref{AlmostDone} as $k\to \infty$ is
\begin{equation*}
\int_D \vartheta(0,y)\frac{\mathrm w\big((0,y),\hat h\hat g_0^{-1}\big)}{|\det\big[JX_{\a_{\wp(\hat h^{-1})}\circ f_0}(0)-\Id\big]|}\,\kappa dy.
\end{equation*}
To rewrite this in invariant terms we begin by noting that $(0,y)$ corresponds to a point $p=\a_{\wp(\hat g)}(p_0)$ with a unique $\hat g\in \hat D$  and that $\kappa dy=d\mu_{\hat G_0}$. The denominator is the absolute value of the determinant of 
\begin{equation*}
(\a_{\wp(\hat h^{-1})}\circ f_0)^*-\Id:N^*_{p_0}\Y\to N^*_{p_0}\Y.
\end{equation*}
which is equal to that of 
\begin{equation*}
(\a_{\wp(\hat h^{-1})}\circ f_0)^*-\Id:N^*_{\a_{\wp(\hat g)}(p_0)}\Y\to N^*_{\a_{\wp(\hat g)}(p_0)}\Y, \quad \hat g\in \hat D
\end{equation*}
for our previously fixed $\hat h\in \preisotropy$. Integrating over $\preisotropy$ with respect to $\mu_\isotropy$ and recalling that $\mu=\mu_\isotropy\otimes \mu_{G_0}$ yields
\begin{multline*}
\lim_{k\to\infty} \langle \iota_\Delta^* K_{\Phi\circ \Av,k},\vartheta w\rangle =\\ \int_{\hat D} \int_{\preisotropy} \frac{\vartheta(\a_{\wp(\hat g)}(p_0))\mathrm w(\a_{\wp(\hat g)}(p_0),\hat h \hat g_0^{-1})}{|\det\big[(\a_{\wp(\hat g_0\hat h^{-1})}\circ f)^*\big|_{N^*_{\a_{\wp(\hat g)}(p_0)}\Y}-\Id\big]|}\, d\hat \mu_\preisotropy (\hat h)\,d\hat \mu_{\hat{G_0}}(\hat g).
\end{multline*}
We now use this formula with $w_0=\calI_\tensor$ which is the case of interest. Replacing $\wp(\hat g_0^{-1})=f(p_0)$, using \eqref{Requivariance} and duality we get
\begin{align*}
\mathrm w(\a_{\wp(\hat g)}(p_0),\hat h\hat g_0^{-1})&= \big\langle (I\otimes \A_{\hat h\hat g_0^{-1}}^*)(\phi_\tensor(\a_{\wp(\hat g)}(p_0),\a_{\wp(\hat g)}(f(p_0)))), \calI_\tensor(\a_{\wp(\hat g)}(p_0))\big\rangle\\
&=\big\langle (I\otimes \A_{\hat h\hat g_0^{-1}}^*)(\phi_\tensor(p_0,f(p_0))), \A_{\wp(\hat g)}^*\otimes \A_{\wp(\hat g^{-1})}\calI_\tensor(\a_{\wp(\hat g)}(p_0))\big\rangle\\
&=\big\langle (I\otimes \A_{\hat h\hat g_0^{-1}}^*)(\phi_\tensor(p_0,f(p_0))),\calI_\tensor(p_0)\big\rangle.
\end{align*}
which is independent of $\hat g\in \hat G_0$. The element $\phi_\tensor(p_0,f(p_0))$ is the tensor representation of the map $\phi(p_0):F_{f(p_0)}\to F_{p_0}$ at $p_0$, $(I\otimes \A_{\hat h\hat g_0^{-1}}^*)(\phi_\tensor(p_0,f(p_0)))$ is the tensor representation of
\begin{equation*}
\phi(p_0)\circ \A_{\hat h\hat g_0^{-1}}:F_{p_0}\to F_{p_0},
\end{equation*}
therefore 
\begin{equation*}
\mathrm w(\a_{\wp(\hat g)}(p_0),\hat h\hat g_0^{-1})=\tr(\phi(p_0)\circ \A_{\hat h\hat g_0^{-1}}).
\end{equation*}

Let now $\hat g_j\in \hat G_0$, $j=1,\dots,N$,  be such that the sets $\calU_j=\a_{\wp(\hat g_j)}(\calU')$ cover $\Y$. Starting from a locally finite partition of unity subordinate to the open set $\bigcup \calU_j'\subset \N$ we get smooth  functions $\vartheta_j$ supported on $\calU_j'$ such that $\sum \vartheta_j=1$ on $\Y$. We then obtain 

\begin{theorem}\label{TheoremC}
Let $\mathfrak k$ is the number of sheets of the covering $\wp|_{\hat G_0}:\hat G_0\to \Y$. Let $\vartheta\in C^\infty(\N)$ be supported in a neighborhood of $\Y$ disjoint all other elements of $\Fix(\pmb f)$ with $\vartheta(p)=1$ if $p\in \Y$. Then
\begin{equation*}
\lim_{k\to\infty} \langle \iota_\Delta^* K_{\Phi\circ \Av,k},\vartheta\phi_\tensor\rangle = \frac{\mu_{\hat G_0}(\hat G_0)}{\mathfrak k}\int_{\preisotropy} \frac{\tr(\phi(p_0)\circ \A_{\hat h \hat g_0^{-1}})}{\big|\det\big[(\a_{\wp(\hat g_0\hat h^{-1})}\circ f)^*\big|_{N^*_{p_0}\Y}-\Id\big]\big|}\, d\hat \mu_\preisotropy (\hat h)
\end{equation*}
where $p_0\in \Y$ is arbitrary and $\hat g_0\in \hat G_0$ is such that $a_{\wp(\hat g_0)}\circ f(p_0)=p_0$.
\end{theorem}


\appendix
\section{Curvature and compact groups}\label{CurvatureAndGroups}

We return to the context and notation of Section \ref{ConnectionsGroupsAverage} to  discuss the existence of connections with curvature valued in $F\otimes \Wedge^2\Hor^*$. 

\begin{proposition}\label{GoodConnection}
Suppose $\hat {\mathcal G}=\set{\A_t:F\to F:t\in \R}$ is a smooth one-parameter group of unitary morphisms with $\A_t$ covering $\a_t$. If the closure $\hat G$ of $\hat {\mathcal G}$ in the compact-open topology of continuous vector bundle homeomorphisms $F\to F$ is a compact Lie group acting smoothly on $F$, then there is a Hermitian connection on $F$ with curvature valued in $F\otimes \Wedge^2\Hor^*$ for which $t\mapsto \A_t$ is parallel transport along $t\mapsto \a_t$.
\end{proposition}

\begin{proof}
The group $\hat G$ is isomorphic to a torus. As before, we let $\A_{\hat g}$ denote the morphisms associated to $\hat g$ but still write $\A_t$ for the originally given morphisms. Using that $\A_t$ covers $\a_t$ and density we again get a surjective map $\wp:\hat G\to G$ such that \eqref{wp} commutes. 

Let $\Lie_\T:C^\infty(\N;F)\to C^\infty(\N;F)$ be defined by 
\begin{equation*}
(\Lie_\T u)(p)=\frac{d}{dt}\Big|_{t=0}\A_{-t}(u(\a_t(p))).
\end{equation*}

There is a Hermitian connection $\nabla^\one$ on $F$ for which the given family of unitary morphisms, $\A_t:F_p\to F_{\a_t p}$, is parallel transport along the curve $t\mapsto \a_t(p)$. Namely pick an arbitrary Hermitian connection $\nabla'$ and define the operator  by 
\begin{equation*}
\nabla^\one:C^\infty(\N;F)\to C^\infty(\N;F\otimes T^*\N),\quad \nabla^\one u=\nabla'u-\nabla'_\T u\otimes \theta+\Lie_\T u\otimes \theta,
\end{equation*}
where $\theta$ is the one-form $\T\contr \gg$ defined in \eqref{PseudoConnection}. It is easy to check that $\nabla^\one$ is again a Hermitian connection, and as easily seen,
\begin{equation}\label{HorizontalIsLie}
\nabla^\one_\T u = \Lie_\T u.
\end{equation}
This implies that $\A_t$ is parallel transport along $\a_t$.

If $u$ is a section of $F$, then $\pmb \A_{\hat g} u$ will mean the section
\begin{equation*}
p\mapsto \A_{\hat g}\big (u(\a_{\wp(\hat g)}^{-1}(p))\big),
\end{equation*}
thus $(\pmb \A_{\hat g}(u))(\a_{\wp(\hat g)}(p))=\A_{\hat g}(u(p))$. 
Define 
\begin{equation*}
\calA_{\hat g}:F\otimes T^*\N\to F\otimes T^*\N,\quad \calA_{\hat g}=\A_{\hat g}\otimes \a_{{\wp(\hat g)}^{-1}}^*;
\end{equation*}
this is a bundle morphism covering $\a_{\wp(\hat g)}$. The corresponding map on sections is
\begin{equation*}
\pmb {\calA}_{\hat g}:C^\infty(\N;F\otimes \T^*\N)\to C^\infty(\N;F\otimes \T^*\N), \quad \pmb {\mathcal A}_{\hat g}(u)(p)=\mathcal A_{\hat g}(u(\a_{\wp(\hat g^{-1})}(p))
\end{equation*}
 Using the connection $\nabla^\one$ constructed above, let
\begin{equation*}
\nabla^{\hat g}=\pmb \calA_{\hat g}\circ\nabla^ \one\circ\pmb \A_{\hat g^{-1}},\quad \hat g\in \hat G.
\end{equation*}
This is again a Hermitian connection. 
The simplest proof is using frames. Let $\eta_\mu$ be a frame of $F$ over some open set $\calU$ with connection forms $\omega^\nu_\mu$. Let $\tilde \eta_\mu=\A_{\hat g}\eta_\mu$, a frame over $\a_{\wp(\hat g)}(\calU)$.  Then
\begin{equation*}
\nabla^\one \A_{\hat g^{-1}}\tilde \eta_\mu=\nabla^\one \eta_\mu=\eta_\rho\otimes \omega^\rho_\mu
\end{equation*}
so 
\begin{equation*}
\nabla^{\hat g}\tilde \eta_\mu= \calA_{\hat g} \nabla^\one \A_{\hat g^{-1}}\tilde \eta_\mu=\A_{\hat g}\eta_\rho\otimes \a_{\wp(\hat g)^{-1}}^*\omega^\rho_\mu=\tilde \eta_\rho\otimes \a_{\wp(\hat g)^{-1}}^*\omega^\rho_\mu.
\end{equation*}
On the other hand, since $\A_{\hat g}$ is unitary, $h_p(\tilde \eta_\mu,\tilde \eta_\nu)=h_{\a_{\wp(\hat g^{-1})}p}(\eta_\mu,\eta_\nu)$, that is, $h(\tilde \eta_\mu,\tilde \eta_\nu)=\a_{\wp(\hat g)^{-1}}^*h(\eta_\mu,\eta_\nu)$, so, since $\nabla^\one$ is Hermitian,
\begin{align*}
dh(\tilde \eta_\mu,\tilde \eta_\nu)=\a_{\wp(\hat g)^{-1}}^*dh(\eta_\mu,\eta_\nu)
&=\a_{\wp(\hat g)^{-1}}^*[h(\nabla^\one \eta_\mu,\eta_\nu)+h(\eta_\mu,\nabla^\one\eta_\nu)] \\
&=\a_{\wp(\hat g)^{-1}}^*[h(\eta_\rho,\eta_\nu)\omega^\rho_\mu+h(\eta_\mu,\eta_\rho)\overline \omega ^\rho_\nu]\\
&=h(\tilde \eta_\rho,\tilde \eta_\nu)\a_{\wp(\hat g)^{-1}}^*\omega^\rho_\mu+h(\tilde \eta_\mu,\tilde \eta_\rho)\a_{\wp(\hat g)^{-1}}^*\overline \omega ^\rho_\nu\\
&=h(\nabla^{\hat g} \tilde \eta_\mu,\tilde \eta_\nu)+h(\tilde \eta_\mu,\nabla^{\hat g}\tilde \eta_\nu).
\end{align*}

We also have $\nabla^{\hat g}_\T=\Lie_\T$. Namely, if $u\in C^\infty(\N;F)$, then 
\begin{multline*}
\nabla^{\hat g}_\T u = \langle (\A_{\hat g}\otimes \a_{{\wp(\hat g)}^{-1}}^*)\circ\nabla^\one \A_{\hat g}^{-1} u,\T\rangle=\A_{\hat g}\circ\nabla^\one_{d\a_{{\hat g}^{-1}}\T} \A_{\hat g}^{-1} u\\
=\A_{\hat g}\circ\nabla^\one_\T \A_{\hat g}^{-1} u= \A_{\hat g}
\Lie_\T\A_{\hat g}^{-1} u=\Lie_\T u
\end{multline*}
using $\Lie_\T\A_{\hat g}^{-1}=\A_{\hat g}^{-1}\Lie_\T$. Thus $t\mapsto \A_t$ is paralel transport along $t\mapsto \a_t$ also relative to $\nabla^{\hat g}$. Also, clearly
\begin{equation*}
\pmb \calA_{\hat g'}\circ\nabla^{\hat g}\circ\pmb \A_{\hat g'}^{-1}=\nabla^{\hat g\hat g'}
\end{equation*}
for any $\hat g,\hat g'\in \hat G$. 

Let $\hat \mu$ denote the normalized Haar measure of $\hat G$. Given $u\in C^\infty(\N;F)$, let 
\begin{equation*}
(\nabla u)(p)=\int_{\hat G} (\nabla^{\hat g}u)(p)\,d\hat \mu(\hat g).
\end{equation*}
This is: $\hat g\mapsto (\nabla^{\hat g}u)(p)$ is a map $\hat G\to F_p\otimes T_p^*\N$, and its integral is defined to be $(\nabla u)(p)$. We show that $\nabla$ is a connection. The map 
\begin{equation*}
\N\times \hat G \ni (p,\hat g)\mapsto (\nabla^{\hat g}u)(p) \in F\otimes T^*\N
\end{equation*}
is smooth, so $\nabla u$ is a smooth section of $F\otimes T^*\N$ if $u$ is smooth. Furthermore, if $f\in C^\infty(\N)$, then
\begin{equation*}
\nabla^{\hat g}(f u)=f\nabla^{\hat g}(u)+u\otimes df,
\end{equation*}
so
\begin{align*}
(\nabla fu)(p)&=\int_{\hat G} \big((f(p)\nabla^{\hat g}u)(p)+u(p)\otimes df(p)\big)\,d\hat \mu(\hat g)\\
&=f(p)(\nabla u)(p)+u(p)\otimes df(p).
\end{align*}
Thus $\nabla$ is indeed a connection. It inherits from the $\nabla^{\hat g}$ the property that $\nabla_\T=\Lie_\T$. Also, again with smooth $u$,
\begin{multline*}
\pmb \calA_{\hat g'}(\nabla u)(\a_{\wp(\hat g')}(p))
=\calA_{\hat g'}(\nabla u(p))
=\calA_{\hat g'}\big( \int_{\hat G} (\nabla^{\hat g}u)(p)\,d\hat \mu(\hat g)\big)\\
= \int_{\hat G} \calA_{\hat g'}\big((\nabla^{\hat g}u)(p)\big)\,d\hat \mu(\hat g)
= \int_{\hat G} \calA_{\hat g'}\big((\nabla^{\hat g}\pmb \A_{\hat g'}^{-1}\pmb \A_{\hat g'}u)(p)\big)\,d\hat \mu(\hat g)\\
= \int_{\hat G} \big( (\nabla^{\hat g\hat g'}\pmb \A_{\hat g'}u)(\a_{\wp(\hat g')}(p))\big)\,d\hat \mu(\hat g)
= \int_{\hat G} \big( (\nabla^{\hat g}\pmb \A_{\hat g'}u)(\a_{\wp(\hat g')}(p))\big)\,d\hat \mu(\hat g)\\
=(\nabla\pmb \A_{\hat g'}u )(\a_{\wp (\hat g')(p)})
\end{multline*}

In particular,
\begin{equation*}
\nabla u(p)=(\pmb \calA_{-t}\nabla\pmb \A_t u)(p).
\end{equation*}

Let $\eta_\mu$ be a frame of $F$ near $p_0$ such that $\Lie_\T \eta_\mu=0$. So $\eta_\mu(\a_t p)=\A_t(\eta_\mu(p))$ for small $t$ and $p$ near $p_0$. Let the $\omega^\nu_\mu$ denote the connection forms of $\nabla$ with respect to this frame. Then
\begin{multline*}
\sum_\nu \eta_\nu\otimes \omega^\nu_\mu=\nabla \eta_\mu=  \calA_{-t} \nabla \A_t\eta_\mu=\calA_{-t} \nabla \eta_\mu\\=\sum_\nu (\A_{-t}\otimes \a_t^*)(\eta_\nu\otimes \omega^\nu_\mu)=\sum_\nu \A_{-t}\eta_\nu\otimes \a_t^*\omega^\nu_\mu = \sum_\nu\eta_\nu\otimes \a_t^*\omega^\nu_\mu.
\end{multline*}
It follows that $\omega^\nu_\mu=\a_t^*\omega^\nu_\mu$ for all $\nu,\mu$. Thus $\Lie_\T\omega^\nu_\mu=0$, which implies $\inner_\T d\omega^\nu_\mu=0$ for all $\nu,\mu$. Thus $\inner_\T \Omega=0$. 
\end{proof}


\section*{}

\textbf{Acknowledgments.} Hartmann would like to thank the hospitality of the Department of Mathematics at Temple University for its hospitality during several visits. Likewise, Mendoza thanks the hospitality of the Departamento de Matem\'atica of Universidade Federal de S\~ao Carlos.


\bibliography{SingularAtiyahBott}
\bibliographystyle{amsplain}

\end{document}